\documentclass[12pt,a4paper, oneside]{amsart}

\usepackage{amsmath, nicefrac, amsthm, verbatim, amsfonts, mathtools, amssymb, upgreek, xcolor, bbm}
\usepackage{graphics, xspace, enumerate}
\usepackage{stix}
\usepackage{dsfont}%this is for nicer double-stroke letters
\usepackage[a4paper,margin=2.5cm]{geometry}%this is for easy margins
\usepackage[bb=boondox]{mathalfa}

\usepackage{graphicx}
\usepackage[colorlinks=true,citecolor=red,urlcolor=blue,linkcolor=red,bookmarksopen=true,unicode=true,pdffitwindow=true]{hyperref}
\usepackage[english]{babel}
\usepackage[languagenames,fixlanguage]{babelbib}
\hypersetup{pdfauthor={}}
\hypersetup{pdftitle={}}

\usepackage{seqsplit,cleveref}

\theoremstyle{plain}
\newtheorem{theorem}{Theorem}[section]

\newtheorem{lemma}[theorem]{Lemma}
\newtheorem{proposition}[theorem]{Proposition}
\theoremstyle{definition}
\newtheorem{remark}[theorem]{Remark}

\newcommand {\Prob} {\ensuremath{\mathbb{P}}}
\newcommand {\R} {\ensuremath{\mathbb{R}}}
\newcommand {\ZZ} {\ensuremath{\mathbb{Z}}}
\newcommand {\N} {\ensuremath{\mathbb{N}}}

\newcommand{\process}[1]{\{#1(x,i;t)\}_{t\geq0}}

\newcommand{\df}{\coloneqq}

\newcommand{\E}{\mathrm{e}}
\newcommand{\Id}{\mathbb{I}}

\newcommand{\B}{\mathsf{B}}
\newcommand{\bb}{\mathrm{b}}
\newcommand{\cc}{\mathrm{c}}
\newcommand{\PP}{\mathcal{P}}
\newcommand{\p}{\mathsf{p}}

\newcommand{\X}{\mathrm{X}}
\newcommand{\A}{\mathcal{A}}

\newcommand{\D}{\mathrm{d}}
\newcommand{\DD}{\mathrm{D}}

\numberwithin{equation}{section}

\title[Periodic homogenization of a class of weakly coupled systems of linear PDE\MakeLowercase{s} ]{Periodic homogenization of a class of   weakly coupled systems of linear PDE\MakeLowercase{s}}

\author[N.\ Sandri\'{c}]{Nikola Sandri\'{c}}
\address[Nikola\ Sandri\'{c}]{Department of Mathematics\\University of Zagreb\\ Zagreb\\Croatia}
\email{nsandric@math.hr}


\begin{document}
\allowdisplaybreaks[4]

\begin{abstract} In this article, basing upon probabilistic methods, we discuss periodic homogenization of a class of   weakly coupled systems of linear   elliptic and parabolic partial differential equations. 
	Under the assumption that the systems have rapidly periodically oscillating coefficients,  we first prove that the appropriately centered  and scaled continuous component of the associated regime switching diffusion process converges weakly to a Brownian motion with covariance matrix given in terms of the coefficients of the systems. The homogenization results  then follow
by employing probabilistic representation of the solutions to the systems and the continuous mapping theorem.
	The presented results generalize the  well-known results related to periodic homogenization of the classical elliptic boundary-value problem and the
classical parabolic initial-value problem for a single equation.
	
\end{abstract}

\maketitle

	\section*{Acknowledgements}
Financial support through the \textit{Alexander-von-Humboldt Foundation} (under project No. HRV 1151902 HFSTE) and \textit{Croatian Science Foundation} (under project No. 8958) 
is  gratefully acknowledged. 

\section{Introduction}\label{S1} 

One of the  classical directions in the analysis of partial differential equations (PDEs) centers around their homogenization (averaging) properties. In this article,  we discuss periodic homogenization of a class of   weakly coupled systems of linear   elliptic and parabolic PDEs.  
	Let $\{\mathcal{L}_i^{\varepsilon}\}_{i\in[n]}$,  $\varepsilon>0$, be a family of second-order elliptic differential operators  of the form
$$
\mathcal{L}^{\varepsilon}_i\,=\,
2^{-1}\mathrm{Tr}\bigl(\upsigma(\cdot/\varepsilon,i)\upsigma(\cdot/\varepsilon,i)^\mathrm{T}\,\nabla\nabla^\mathrm{T}\bigr)+
\bigl(\varepsilon^{-1}\bb(\cdot/\varepsilon,i)^{\mathrm{T}}+\cc(\cdot/\varepsilon,i)^{\mathrm{T}}\bigr)\nabla\,,
$$ and let $\mathcal{Q}=(\mathrm{q}_{ij})_{i,j\in[n]}$ be an $n\times n$ real matrix with $\mathrm{q}_{ij}\ge0$ for $i\neq j$. If $n=1$, we put $\mathcal{Q}\df 0$. For $f:[n]\to\R$ and $i\in[n]$, define $\mathcal{Q}f(i)\df\sum_{j\in[n]}\mathrm{q}_{ij}f(j).$
The main goal of this article is to discuss periodic homogenization (that is, asymptotic behavior of the solution as $\varepsilon\to0$) for the following weakly coupled  system (the system is coupled in the terms which are not differentiated) of elliptic boundary-value problems \begin{equation}\label{ES1.2}
\begin{aligned}
\bigl(\mathcal{L}_i^{\varepsilon}+\varepsilon^{-2}\mathcal{Q}\bigr) u^\varepsilon(x,i)+
e(x/\varepsilon,i)\,u^{\varepsilon}(x,i)+f(x)&\,=\,0\,,\qquad (x,i)\in\mathscr{D}\times[n]\,,\\
u^\varepsilon(x,i)&\,=\, g(x)\,,\qquad (x,i)\in\partial\mathscr{D}\times[n]\,,
\end{aligned}
\end{equation}
as well as the  coupled system of parabolic initial-value problems
\begin{equation}\label{ES1.3} 
\begin{aligned}
\partial_tu^\varepsilon(t,x,i)\,=\,&\bigl(\mathcal{L}_i^{\varepsilon}+\varepsilon^{-2}\mathcal{Q}\bigr) u^\varepsilon(t,x,i)\\ &+e(x/\varepsilon,i)u^{\varepsilon}(t,x,i)+f(x)\,,\qquad (t,x,i)\in(0,\infty)\times\R^d\times[n]\,,\\
u^\varepsilon(0,x,i)\,=\,& g(x)\,,\qquad (x,i)\in\R^d\times[n]\,.
\end{aligned}
\end{equation}
Without loss of generality, in the sequel we assume $\mathrm{q}_{ii}=-\sum_{j\neq i}\mathrm{q}_{ij}$; otherwise we replace $e(x/\varepsilon,i)$ in \cref{ES1.2,ES1.3} by $e(x/\varepsilon,i)+\varepsilon^{-2}\sum_{j\in[n]}\mathrm{q}_{ij}$. 
These problems are a generalization of  the classical elliptic boundary-value problem and the classical parabolic initial-value problem for a single equation ($n=1$). Typical examples of such systems are the Maxwell's equations, Hamilton--Jacobi equations and semiconductor equations, see \cite{Azunre-2017} and \cite{Engler-Lenhart-1991} and the references therein.

Our approach   relies on probabilistic techniques.  By combining classical PDE results (existence of a smooth solution to the corresponding Poisson equation) and the fact that the underlying regime switching diffusion process $\{(\X^\varepsilon(x,i;t),\Lambda(i;t/\varepsilon^2))\}_{t\ge0}$ associated to  $\mathcal{L}_i^{\varepsilon}+\varepsilon^{-2}\mathcal{Q}$  does not show a singular behavior in its motion (that is, it is irreducible),  
 we first show that  the (appropriately centered) process $\{\X^\varepsilon(x,i;t)\}_{t\ge0}$ satisfies a functional CLT with Brownian limit as $\varepsilon\to0$ (see \Cref{T3.1}). 
 The homogenization results (see \Cref{T4.1,T4.2}) then follow 
  by employing probabilistic representation (the Feynman-Kac formula) of the  solutions to the problems in \cref{ES1.2,ES1.3}, obtained in  \cite[Theorems 3.2 and 3.3]{Zhu-Yin-Baran-2015}, and the continuous mapping theorem.  This idea goes back to M. I. Fre\u{\i}dlin  \cite{Freidlin-1964} (see also \cite[Chapter 3]{Bensoussan-Lions-Papanicolaou-Book-1978}).

\subsection{Literature review}  
Our work contributes to the classical theory of periodic homogenization. Most of the existing literature on this subject focuses on  homogenization of   problems with a single equation; for instance, see the classical monographs \cite{Allaire-2002-Book}, \cite{Bensoussan-Lions-Papanicolaou-Book-1978},   \cite{Jikov-Kozlov-Oleinik-1994-Book} and  \cite{Tartar-2009-Book}. In this article, we discuss homogenization of weakly coupled systems of linear PDEs, that is, systems with no coupling in the derivatives. 
In a certain sense, these systems form a most general class of systems satisfying the maximum principle (see \cite[Chap.~3 Sec.~8]{Protter-Weinberger-Book-1984}), and as such have drawn attention of many authors. Also, there is a direct connection between   these systems and 
probability theory. Namely, the operator $\mathcal{L}_i^\varepsilon+\varepsilon^{-2}\mathcal{Q}$ is the infinitesimal generator of a
diffusion process with random switching mechanism (see \cite{Skorokhod-Book-1989}). This makes them  interesting both to analysts and probabilists.
Solutions to the  problems in \cref{ES1.2,ES1.3} were studied in \cite{Chen-Zhao-1994}, \cite{Chen-Zhao-1996}, \cite{Eizenberg-Freidlin-1990}, \cite{Eidelman-Book-1969} and \cite{Friedman-Book-1964}. Their Schauder regularity can be found in 
\cite{Addona-Lorenzi-2023} and \cite{Delmonte-Lorenzi-2011}, and
their Feynman-Kac representation is established in  \cite{Zhu-Yin-Baran-2015} and \cite{Freidlin-Book-1985}.
Potential theory for these systems was developed in \cite{Chen-Zhao-1996} while strong positivity
results can be found in \cite{Chen-Zhao-1996} and \cite{Sweers-1992}. Maximum principle is established in \cite{Boyadzhiev-Kutev-2018} and \cite{Sirakov-2009},  results on  Harnack inequality  in \cite{Arapostathis-Ghosh-Marcus-1999} and \cite{Chen-Zhao-1997}, and existence and characterization of all the systems of invariant measures of the associated semigroups are obtained in  \cite{Addona-Angiuli-Lorenzi-2019}. Homogenization of the problems of type in \cref{ES1.2,ES1.3} was discussed in \cite{Eizenberg-Freidlin-1990},  \cite{Eizenberg-Freidlin-1993-PTRF}, \cite{Eizenberg-Freidlin-1993-AOP}  and \cite{Pinsky-Pinsky-1993}. However, observe that the  systems considered  in these works are not periodic. %In \cite{Eizenberg-Freidlin-1990},  \cite{Eizenberg-Freidlin-1993-PTRF} and  \cite{Eizenberg-Freidlin-1993-AOP}  homogenization of an elliptic system is discussed, while in \cite{Pinsky-Pinsky-1993} a parabolic system is treated. In the later work the authors assume that $\upsigma(x,i)\upsigma(x,i)^\mathrm{T}\equiv\Id_d$ and $\bb(x,i)=|x|^\delta\bar\bb(x/|x|,i)$ for some $\delta\in[-1,1)$ and continuously differentiable function $\bar\bb(x,i)$ (in the second variable) on $(d-1)$-dimensional unit sphere.
Let us also mention that various
probabilistic phenomena associated with this kind of systems have recently been 
investigated, for instance, see the  classical references in the field \cite{Mao-Yuan-Book-2006} and \cite{Yin-Zhu-Book-2010}.

\subsection{Notation} 
We summarize some  notation used throughout the article. For $n\in\N$, the symbol $[n]$ stands for the set $\{1,\dots,n\}$. 
We use $\R^d$, $d\in\N$, to denote real-valued
$d$-dimensional  vectors, and write $\R$ for $d=1$. All  vectors will be column vectors. The Euclidean norm on $\R^d$ is denoted by $\lvert\cdot\rvert$.
By $M^{\mathrm{T}}$ and $\|M\|_{HS}\df (\mathrm{Tr}MM^{\mathrm{T}})^{1/2}$ we denote the transpose and the Hilbert-Schmidt norm of  a  matrix $M$, respectively. For a 
square matrix $M$,
$\mathrm{Tr}\, M$ stands for its trace. The $d\times d$ identity matrix is denoted by $\mathbb{I}_d$. For a set $A\subseteq\R^d$,  the symbols
$\mathbb{1}_{A}$, $\overline{A}$ and $\partial A$  stand for
the indicator function,  (topological) closure and (topological) boundary of $A$, respectively. The open ball of radius $\rho>0$ around $x\in\R^d$ is denoted by $\mathscr{B}_\rho(x)$.
For topological spaces $(T,\mathcal{T})$ and $(S,\mathcal{S})$ we let  $\mathfrak{B}(T)$ and $\mathcal{B}(T,S)$ denote the Borel $\sigma$-algebra on $(T,\mathcal{T})$ and the space of $\mathfrak{B}(T)/\mathfrak{B}(S)$-measurable functions, respectively. Also, for $A\subseteq T$, $\mathfrak{B}(A)$ stands for $\{A\cap B\colon B\in\mathfrak{B}(T)\}$.
For $f\in\mathcal{B}(\R^d,\R^p)$ we let $\|f\|_\infty\df\sup_{x\in\R^d}|f(x)|$ denote its supremum norm, and 
$\mathcal{B}_b(\R^d,\R^p)$ stands for $\{f\in\mathcal{B}(\R^d,\R^p)\colon \|f\|_\infty<\infty\}$. We use  $\mathcal{C}_b^k(\R^d,\R^p)$,  $\mathcal{C}_{u,b}^k(\R^d,\R^p)$ and $\mathcal{C}^k_\infty(\R^d,\R^p)$, $k\in\N_0\cup\{\infty\}$, to denote the subspace of $\mathcal{B}_b(\R^d,\R^p)\cap \mathcal{C}^k(\R^d,\R^p)$ of all  $k$ times differentiable functions such that all derivatives up to order $k$  are bounded,  uniformly continuous and bounded, and vanish at infinity, respectively.
 By $\mathcal{B}(\R^d\times[n],\R^p)$, $\mathcal{B}_b(\R^d\times[n],\R^p)$,  $\mathcal{C}^k(\R^d\times[n],\R^p)$, $\mathcal{C}_b^k(\R^d\times[n],\R^p)$, $\mathcal{C}_{u,b}^k(\R^d\times[n],\R^p)$ and $\mathcal{C}^k_\infty(\R^d\times[n],\R^p)$, $k\in\N_0\cup\{\infty\}$, we denote the spaces of functions $f:\R^d\times[n]\to\R^p$ such that $f(\cdot,i)\in \mathcal{B}(\R^d,\R^p),$ $f(\cdot,i)\in \mathcal{B}_b(\R^d,\R^p),$  $f(\cdot,i)\in \mathcal{C}^k(\R^d,\R^p)$, $f(\cdot,i)\in \mathcal{C}_b^k(\R^d,\R^p)$, $f(\cdot,i)\in \mathcal{C}_{u,b}^k(\R^d,\R^p)$ and $f(\cdot,i)\in \mathcal{C}_\infty(\R^d,\R^p)$,  respectively, for every $i\in[n]$.
Gradient of $f\in\mathcal{C}^1(\R^d\times[n],\R)$ is denoted by $\nabla f(x,i)=(\partial_1 f(x,i),\dots,\partial_d f(x,i))^\mathrm{T}$, and for $f=(f_1,\dots,f_p)^\mathrm{T}\in\mathcal{C}^1(\R^d\times[n],\R^p)$ we write $\mathrm{D} f(x,i)=(\nabla f_1(x,i),\dots,\nabla f_p(x,i))^\mathrm{T}$  for the corresponding Jacobian. 
For a Borel  measure $\upmu$
on $\mathfrak{B}(\R^d\times[n])$ and $f=(f_1,\dots,f_p)^{\mathrm{T}}\in\mathcal{B}(\R^d\times[n],\R^p)$, we often use the convenient notation
$\upmu(f)=\int_{\R^d\times[n]} f(x,i)\,\upmu(\D{x}\times\{i\})\df (\int_{\R^d\times[n]} f_1(x,i)\,\upmu(\D{x}\times\{i\}),\dots,\int_{\R^d\times[n]} f_p(x)\,\upmu(\D{x}\times\{i\}))^\mathrm{T}$.
For $\tau=(\tau_1,\dots, \tau_{d})^{\mathrm{T}}\in (0,\infty)^{d}$, we let $\ZZ^{d}_\tau\df\{(\tau_1 k_1,\dotsc,\tau_{d} k_d)^\mathrm{T}\colon (k_1,\dotsc,k_d)^\mathrm{T}\in\ZZ^d\},$
and, for  $x\in\R^{d}$, 
$$[x]_\tau\,\df\,\bigl\{y\in\R^{d}\colon x-y\in\ZZ_\tau^{d}\bigr\}\,,\qquad\textrm{and}\qquad
\mathbb{T}^{d}_\tau\times[n]\,\df\,\bigl\{([x]_\tau,i)\colon (x,i)\in\R^{d}\times[n]\bigr\}\,.$$
Clearly,
$\mathbb{T}^{d}_\tau\times[n]$ is obtained
by identifying the opposite
faces of $[0,\tau]\times[n]\df[0,\tau_1]\times\cdots\times[0,\tau_{d}]\times[n]$. 
The corresponding Borel $\sigma$-algebra is denoted by $\mathfrak{B}(\mathbb{T}_\tau^d\times[n])$, which can be identified with the sub-$\sigma$-algebra of $\mathfrak{B}(\R^d\times[n])$  of sets of the form $\bigcup_{k_\tau\in\ZZ_\tau^d}\{(x+k_\tau,i)\colon x\in B_\tau\}$, $B_\tau\in\mathfrak{B}([0,\tau])$ and $i\in[n]$.
The covering map $\R^d\times[n]\ni (x,i)\mapsto ([x]_\tau,i)\in\mathbb{T}_\tau^d\times[n]$ is denoted by 
$\Pi_{\tau}(x,i)$.
A
function $f:\R^{d}\times[n]\to\R^{p}$ is called $\tau$-periodic if
$$f(x+k_\tau,i )\,=\,f(x,i)\qquad \forall\, (x,k_\tau,i)\in\R^{d}\times\ZZ^{d}_\tau\times[n]\,.$$
Clearly, every $\tau$-periodic function $f(x,i)$ is completely and uniquely determined by its restriction $f|_{[0,\tau]\times[n]}(x,i)$ to $[0,\tau]\times[n]$, and since  $f|_{[0,\tau]\times[n]}(x,i)$ assumes the same value on opposite faces of $[0,\tau]\times[n]$ it can be identified  by a function $f_\tau:\mathbb{T}^{d}_\tau\times[n]\to\R^{p}$ given with $f_\tau ([x]_\tau,i)\df f(x,i).$  Using this identification, in an analogous way as above we define    $\mathcal{C}_b^k(\mathbb{T}_\tau^d\times[n],\R^p)=\mathcal{C}^k_\infty(\mathbb{T}_\tau^d\times[n],\R^p)=\mathcal{C}^k(\mathbb{T}_\tau^d\times[n],\R^p)$, $k\in\N_0\cup\{\infty\}$. 
For notational convenience, we   write $x$ instead of $[x]_\tau$, and $f(x,i)$ instead of $f_\tau(x,i)$.

\subsection{Organization of the article}  In the next section, we first discuss certain structural and ergodic properties of a regime switching diffusion process associated to the operator $\mathcal{L}_i^\varepsilon+\varepsilon^{-2}\mathcal{Q}$. 
Then, in \Cref{S3}, we prove that under an appropriate centering the first (continuous) component of this  process satisfies a functional CLT, as $\varepsilon\to0$,  with a Browninan  limit, which is the key probabilistic argument in discussing the homogenization of the problems in \cref{ES1.2,ES1.3}.
Finally, in \Cref{S4},  we prove the homogenization results.

\section{Structural properties of the associated regime switching diffusion process}\label{S2}

 Let $\{\B(t)\}_{t\ge0}$  be a standard $m$-dimensional Brownian motion (starting from the origin)  and let $\{\Lambda(i;t)\}_{t\ge0}$ be a  temporally-homogeneous Markov chain  with  c\`{a}dl\`{a}g sample paths, state space $[n]$ and infinitesimal generator $\mathcal{Q}$.
The processes  $\{\B(t)\}_{t\ge0}$ and $\{\Lambda(i;t)\}_{t\ge0}$ are   independent and defined on a stochastic basis $(\Omega, \mathcal{F}, \{\mathcal{F}_t\}_{t\ge0},\Prob)$ satisfying the usual conditions. Recall that 
$\mathrm{q}_{ij}=\lim_{t\to 0}\Prob(\Lambda(i;t)=j)/t$ for $i\neq j$. Furthermore, define $\B^\varepsilon(t)\df \varepsilon \B(\varepsilon^{-2}t)$ for $t\ge0$. Clearly,  $\{\B^\varepsilon(t)\}_{t\ge0}\stackrel{(\rm{d})}{=}\{\B(t)\}_{t\ge0}$, although $\{\B^\varepsilon(t)\}_{t\ge0}$ is not a martingale with respect to $\{\mathcal{F}_t\}_{t\ge0}$.  Here, $\stackrel{(\rm{d})}{=}$ denotes the equality in distribution. 

Throughout the article we impose the following assumptions on the  coefficients  $\bb,\cc:\R^d\times[n]\to\R^d$ and $\upsigma:\R^d\times[n]\to\R^{d\times m}$:

\smallskip

\begin{description}
	\item[(A1)]    $\bb,\cc\in \mathcal{C}^2(\mathbb{T}_\tau^d\times[n],\R^d)$ and $\upsigma\in\mathcal{C}^2(\mathbb{T}_\tau^d\times[n],\R^{d\times m})$.
\end{description}

\smallskip

\noindent
According to    \cite{Xi-Yin-Zhu-2019} (see also \cite{Kunwai-Zhu-2020} and \cite{Mao-Yuan-Book-2006}),    \textbf{(A1)} implies that the regime-switching stochastic differential equation: 

\begin{equation}
\begin{aligned}
\label{ES2.1}
\D \X^\varepsilon(x,i;t) &\,=\, \left(\varepsilon^{-1}\bb\bigl(\X^\varepsilon(x,i;t)/\varepsilon,\Lambda(i;t/\varepsilon^2)\bigr)+\cc\bigl(\X^\varepsilon(x,i;t)/\varepsilon,\Lambda(i;t/\varepsilon^2)\bigr)\right)\D t \\&\ \ \ \ \ +\upsigma\bigl(\X^\varepsilon(x,i;t)/\varepsilon,\Lambda(i;t/\varepsilon^2)\bigr)\D  \B^\varepsilon(t)\\ \X^\varepsilon(x,i;0)&\,=\, x \in\R^d\\
\Lambda(i;0)&\,=\, i\in[n]
\end{aligned}
\end{equation}

\smallskip

\begin{itemize}
	\item [(i)]   admits a unique nonexplosive strong solution $\{\X^\varepsilon(x,i;t)\}_{t\ge0}$ which has continuous sample paths
	
\smallskip
	
	\item[(ii)] the process $\{(\X^\varepsilon(x,i;t),\Lambda(i;t/\varepsilon^2))\}_{t\ge0}$ is a temporally-homogeneous strong Markov process with c\`{a}dl\`{a}g sample paths and transition kernel $\p^\varepsilon(t,(x,i),\D y\times \{j\})=\Prob((\X^\varepsilon(x,i;t),$\linebreak $\Lambda(i;t/\varepsilon^2))\in\D y\times \{j\})$
	
\smallskip

	\item[(iii)]	the corresponding semigroup of linear operators $\{\mathcal{P}^\varepsilon_t\}_{t\ge0}$, defined by
	$$\mathcal{P}^\varepsilon_tf(x,i)\,\df\, \int_{\R^d\times[n]}f(y,j)\, \mathsf{p}^\varepsilon\bigl(t,(x,i),\D y\times \{j\}\bigr)\,,\qquad   f\in \mathcal{B}_b(\R^d\times[n],\R)\,,$$ satisfies the 
	$\mathcal{C}_b$-Feller property, that is, $\mathcal{P}^\varepsilon_t(\mathcal{C}_b(\R^d\times[n],\R))\subseteq\mathcal{C}_b(\R^d\times[n],\R)$ for all $t\ge0$
	
\smallskip
	
	\item[(iv)]	for any  $f\in \mathcal{C}^2(\R^d\times[n],\R)$ the process $$\left\{f\bigl(\X^\varepsilon(x,i;t),\Lambda(i;t/\varepsilon^2)\bigr)-f(x,i)-\int_0^t\mathcal{L}^\varepsilon f\bigl(\X^\varepsilon(x,i;s),\Lambda(i;s/\varepsilon^2)\bigr)\D s\right\}_{t\ge0}$$ is a $\mathbb{P}$-local martingale, where
$\mathcal{L}^\varepsilon f(x,i)=\mathcal{L}^\varepsilon_if(x,i)+\varepsilon^{-2}\mathcal{Q}f(x,i)$ 
	
\smallskip
	
	\item[(v)]   for each     $i\in[n]$  the  stochastic differential equation: 
	\begin{align*}
	\D \X^{\varepsilon,i}(x;t) &\,=\, \left(\varepsilon^{-1}\bb\bigl(\X^{\varepsilon,i}(x;t)/\varepsilon,i\bigr)+\cc\bigl(\X^{\varepsilon,i}(x;t)/\varepsilon,i\bigr)\right)\D t+\upsigma\bigl(\X^{\varepsilon,i}(x;t)/\varepsilon,i\bigr)\D  \B^\varepsilon(t)\\  \X^{\varepsilon,i}(x;0)&\,=\,x\in\R^d
	\end{align*} admits a unique nonexplosive strong solution $\{\X^{\varepsilon,i}(x;t)\}_{t\ge0}$ which has continuous sample paths, it is a temporally-homogeneous strong Markov  process with transition kernel $\p^{\varepsilon,i}(t,x,\D y)=\Prob(\X^{\varepsilon,i}(x;t)\in\D y)$, the corresponding semigroup of linear operators $\{\mathcal{P}^{\varepsilon,i}_t\}_{t\ge0}$, defined by
	$$\mathcal{P}^{\varepsilon,i}_tf(x)\,\df\, \int_{\R^d}f(y)\, \p^{\varepsilon,i}(t,x,\D y)\,,\qquad   f\in \mathcal{B}_b(\R^d,\R)\,,$$ satisfies the 
	$\mathcal{C}_b$-Feller property and for any  $f\in \mathcal{C}^2(\R^d,\R)$ the process $$\left\{f\bigl(\X^{\varepsilon,i}(x;t)\bigr)-f(x)-\int_0^t\mathcal{L}^\varepsilon_if\bigl(\X^{\varepsilon,i}(x;s)\bigr)\D s\right\}_{t\ge0}$$ is a $\mathbb{P}$-local martingale.
\end{itemize}

\smallskip

The $\mathcal{B}_b$-infinitesimal generator $(\mathcal{A}^\varepsilon, \mathcal{D}_{\mathcal{A}^\varepsilon})$ of $\{\PP^\varepsilon_t\}_{t\ge0}$ (or of  $\{(X^\varepsilon(x,i;t),\Lambda(i;t/\varepsilon^2))\}_{t\ge0}\}$)  is a linear operator $\mathcal{A}^\varepsilon:\mathcal{D}_{\mathcal{A}^\varepsilon} \to \mathcal{B}_b(\R^d\times[n],\R)$ defined by 
\begin{equation*}
\mathcal{A}^\varepsilon f\,\df\,\lim_{t \to 0} \frac{\PP^\varepsilon_t f-f}{t}\end{equation*} with $$f \in \mathcal{D}_{\mathcal{A}^\varepsilon}\,\df\,\left\{f \in \mathcal{B}_b(\R^d\times[n],\R)\colon \lim_{t \to 0} \frac{\PP^\varepsilon_t f(\cdot,i)-f(\cdot,i)}{t}\ \text{ exists in}\ \lVert\cdot\rVert_{\infty} \right\}\,.$$

\begin{proposition}\label{P2.1} Assume \textbf{(A1)}. Then,  $\mathcal{C}_{u,b}^2(\R^{d}\times[n],\R)\subseteq \mathcal{D}_{\mathcal{A}^\varepsilon}$ and $\A^\varepsilon|_{\mathcal{C}_{u,b}^2(\R^{d}\times[n],\R)}=\mathcal{L}^\varepsilon.$ 
\end{proposition}
\begin{proof} Let $f\in\mathcal{C}_{u,b}^2(\R^{d}\times[n],\R)$.
	By employing generalized It\^{o}'s formula (see for instance \cite[eq. 2.10]{Zhu-Yin-Baran-2015}) we have that 	
	\begin{equation}\label{ES2.3}	\begin{aligned}&\bigl|\bigl(\PP^\varepsilon_tf(x,i)-f(x,i)\bigr)/t-\mathcal{L}^\varepsilon f(x,i)\bigr|\\&\,\le\,\left|\frac{1}{t}\int_0^t\bigl(\PP^\varepsilon_s\mathcal{L}_i^\varepsilon f(x,i)-\mathcal{L}_i^\varepsilon f(x,i)\bigr)\D s\right|
	+\left|\frac{\varepsilon^{-2}}{t}\int_0^t\bigl(\PP^\varepsilon_s\mathcal{Q}f(x,i)-\mathcal{Q}f(x,i)\bigr)\D s\right|\,.\end{aligned}\end{equation}
	Further, let $\uptau^{i}\df\inf\{t\ge0:\Lambda(i,t)\neq i\}$. Recall that $\uptau^{i}\sim\mathrm{Exp}(-\mathrm{q}_{ii})$.	We now have
	\begin{align*}
	&\left|\frac{1}{t}\int_0^t\bigl(\PP^\varepsilon_s\mathcal{L}_i^\varepsilon f(x,i)-\mathcal{L}_i^\varepsilon f(x,i)\bigr)\D s\right|\\&\,\le\,\left|\frac{1}{t}\int_0^t\mathbb{E}\bigl[\bigl(\mathcal{L}_i^\varepsilon f\bigl(\X^\varepsilon(x,i;s),\Lambda(i;s/\varepsilon^2)\bigr)-\mathcal{L}_i^\varepsilon f(x,i)\bigr)\mathbb{1}_{\{\uptau^{i}/\varepsilon^2>t\}}\bigr]\D s\right|\\
	&\ \ \ \ \ +\left|\frac{1}{t}\int_0^t\mathbb{E}\bigl[\bigl(\mathcal{L}_i^\varepsilon f\bigl(\X^\varepsilon(x,i;s),\Lambda(i;s/\varepsilon^2)\bigr)-\mathcal{L}_i^\varepsilon f(x,i)\bigr)\mathbb{1}_{\{\uptau^{i}/\varepsilon^2\le t\}}\bigr]\D s\right|\\
	&\,\le\, \left|\frac{1}{t}\int_0^t\mathbb{E}\bigl[\bigl(\mathcal{L}_i^\varepsilon f\bigl(\X^{\varepsilon,i}(x;s),i\bigr)-\mathcal{L}_i^\varepsilon f(x,i)\bigr)\mathbb{1}_{\{\uptau^{i}/\varepsilon^2>t\}}\bigr]\D s\right|\\
	&\ \ \ \ \  +2\max_{ i\in[n]}\lVert \mathcal{L}_i^\varepsilon f(\cdot,i)\rVert_\infty \max_{ i\in[n]}\Prob(\uptau^{i}/\varepsilon^2\le t)\\
	&\,=\, \left|\frac{1}{t}\int_0^t\bigl(\PP^{\varepsilon,i}_s\mathcal{L}_i^\varepsilon f(x,i)-\mathcal{L}_i^\varepsilon f(x,i)\bigr)\D s\right|\Prob(\uptau^{i}/\varepsilon^2> t) +2\max_{ i\in[n]}\lVert \mathcal{L}_i^\varepsilon f(\cdot,i)\rVert_\infty \max_{ i\in[n]}\Prob(\uptau^{i}/\varepsilon^2\le t)\,,
	\end{align*}
	where the second step follows from \cite[Remark 2.2]{Chen-Chen-Tran-Yin-2019} and in the last step we used independence of $\{\X^{\varepsilon,i}(x;t)\}_{t\ge0}$ and $\{\Lambda(i;t)\}_{t\ge0}$. Observe next that due to $\tau$-periodicity of the coefficients, $\mathcal{L}_i^\varepsilon f\in \mathcal{C}_{u,b}(\R^{d}\times[n],\R)$. Fix $\epsilon>0$ and let $\delta>0$ be such that  $|\mathcal{L}_i^\varepsilon f(x,i)-\mathcal{L}_i^\varepsilon f(y,i)|<\epsilon$ for all $i\in[n]$ and   $x,y\in\R^d$, $|x-y|<\delta$. We have
	\begin{align*}
	&\left|\frac{1}{t}\int_0^t\bigl(\PP^{\varepsilon,i}_s\mathcal{L}_i^\varepsilon f(x,i)-\mathcal{L}_i^\varepsilon f(x,i)\bigr)\D s\right|\\&\,\le\,	\frac{1}{t}\int_0^t\sup_{x\in\R^d}\mathbb{E}\left[\bigl|\mathcal{L}_i^\varepsilon f\bigl(\X^{\varepsilon,i}(x,s),i\bigr) -\mathcal{L}_i^\varepsilon f(x,i)\bigr|\mathbb{1}_{\{|\X^{\varepsilon,i}(x,s)-x|<\delta\}}\right]\D s\\
	&\ \ \ \ \ + 2\|\mathcal{L}_i^\varepsilon f(\cdot,i)\|_\infty \frac{1}{t}\int_0^t \sup_{x\in\R^d}\Prob\bigl(|\X^{\varepsilon,i}(x,s)-x|\ge\delta\bigr)\D s\\
	&\,\le\, \epsilon + 2\|\mathcal{L}_i^\varepsilon f(\cdot,i)\|_\infty \frac{1}{t}\int_0^t \sup_{x\in\R^d}\Prob\bigl(|\X^{\varepsilon,i}(x,s)-x|\ge\delta\bigr)\D s\,.
	\end{align*}	
	Next,
	\begin{align*}&\Prob\bigl(|\X^{\varepsilon,i}(x,s)-x|\ge\delta\bigr)\\&\,\le\,\frac{1}{\delta^2}\mathbb{E}\bigl[|\X^{\varepsilon,i}(x,s)-x|^2\bigr]\\
	&\,=\,\frac{1}{\delta^2}\mathbb{E}\Bigg[\Bigg|\int_0^s\bigl(\varepsilon^{-1}\bb\bigl(\X^{\varepsilon,i}(x,v),i\bigr)+\cc\bigl(\X^{\varepsilon,i}(x,v),i\bigr)\bigr)\D v+\int_0^s\upsigma\bigl(\X^{\varepsilon,i}(x,v),i\bigr)\D\B^\varepsilon(v)\Bigg|^2\Bigg]\\
	&\,\le\,\frac{2\|\varepsilon^{-1}\bb(\cdot,i)+\cc(\cdot,i)\|^2_\infty+2\|\upsigma(\cdot,i)\upsigma(\cdot,i)^{\mathrm{T}}\|_{\mathrm{HS}}^2}{\delta^2}\max\{s,s^2\}\,.
	\end{align*}
	Thus, $$\limsup_{t \to 0}\left\|\frac{1}{t}\int_0^t\bigl(\PP^{\varepsilon,i}_s\mathcal{L}_i^\varepsilon f(x,i)-\mathcal{L}_i^\varepsilon f(x,i)\bigr)\D s\right\|_\infty\,\le\,\epsilon\,,$$ which in turn implies that 
	%	By analogous argumentation as in the case of the semigroup $\{\PP^\varepsilon_t\}_{t\ge0}$ we see that $\{\PP_t^{\varepsilon,i}\}_{t\ge0}$ also preserves the class of $\varepsilon \tau$-periodic functions in $\mathcal{B}_b(\R^d,\R)$.  This, together with $\varepsilon \tau$-periodicity of $\mathcal{L}_i^\varepsilon f(x,i)$, gives
%	\begin{align*}
%	&\left|\frac{1}{t}\int_0^t\bigl(\PP^\varepsilon_s\mathcal{L}_i^\varepsilon f(x,i)-\mathcal{L}_i^\varepsilon f(x,i)\bigr)\D s\right|\\&\,\le\,	\frac{1}{t}\int_0^t\sup_{x\in[0,\varepsilon\tau]}\bigl|\PP_s^{\varepsilon,i}\mathcal{L}_i^\varepsilon f(\cdot,i)-\mathcal{L}_i^\varepsilon f(\cdot,i)\bigr|\D s  +2\max_{ i\in[n]}\lVert \mathcal{L}_i^\varepsilon f(\cdot,i)\rVert_\infty \max_{ i\in[n]}\Prob(\uptau^i/\varepsilon^2\le t)\,.
%	\end{align*}	
%	According to \cite[Corollary 3.3]{Schilling-Schnurr-2010}, $\{\X^{\varepsilon,i}(x;t)\}_{t\ge0}$ is a $\mathcal{C}_\infty$-Feller process (that is,  $\mathcal{P}^{\varepsilon,i}_t(\mathcal{C}_\infty(\R^d,\R))\subseteq\mathcal{C}_\infty(\R^d,\R)$ for all $t\ge0$ and $\lim_{t \to 0}\|\mathcal{P}^{\varepsilon,i}_tf-f\|_\infty=0$ for all $f\in\mathcal{C}_\infty(\R^d,\R)$). Hence, from \cite[Lemma 4.8.7]{Jacob-Book-I-2001} follows that  $$\lim_{t\to 0}\sup_{x\in[0,\varepsilon\tau]}\bigl|\PP_t^{\varepsilon,i}\mathcal{L}_i^\varepsilon f(\cdot,i)-\mathcal{L}_i^\varepsilon f(\cdot,i)\bigr|\,=\,0\,.$$ Consequently,
	$$
	\lim_{t \to 0}\left\lVert\frac{1}{t}\int_0^t\bigl(\PP^\varepsilon_s\mathcal{L}_i^\varepsilon f(\cdot,i)-\mathcal{L}_i^\varepsilon f(\cdot,i)\bigr)\D s\right\rVert_\infty	\,=\,0\,.$$

	Let us now consider the second term in \cref{ES2.3}. Analogously as above, we have
	\begin{align*}
	&\left|\frac{\varepsilon^{-2}}{t}\int_0^t\bigl(\PP^\varepsilon_s\mathcal{Q}f(x,i)-\mathcal{Q}f(x,i)\bigr)\D s\right|\\
	&\,\le\, \left|\frac{\varepsilon^{-2}}{t}\int_0^t\mathbb{E}\bigl[\bigl(\mathcal{Q}f\bigl(\X^\varepsilon(x,i;s),\Lambda(i;s/\varepsilon^2)\bigr)-\mathcal{Q}f(x,i)\bigr)\mathbb{1}_{\{\uptau^i/\varepsilon^2>t\}}\bigr]\D s\right| \\
	&\ \ \ \ \ +\left|\frac{\varepsilon^{-2}}{t}\int_0^t\mathbb{E}\bigl[\bigl(\mathcal{Q}f\bigl(\X^\varepsilon(x,i;s),\Lambda(i;s/\varepsilon^2)\bigr)-\mathcal{Q}f(x,i)\bigr)\mathbb{1}_{\{\uptau^i/\varepsilon^2\le t\}}\bigr]\D s\right|\\
	&\,\le\, 
	\left|\frac{\varepsilon^{-2}}{t}\int_0^t\mathbb{E}\bigl[\bigl(\mathcal{Q}f\bigl(\X^{\varepsilon,i}(x;s),i\bigr)-\mathcal{Q}f(x,i)\bigr)\bigr]\D s\right|\Prob(\uptau^i/\varepsilon^2> t)\\
	&\ \ \ \ \  +2\varepsilon^{-2}\max_{ i\in[n]}\lVert \mathcal{Q}f(\cdot,i)\rVert_\infty \max_{ i\in[n]}\Prob(\uptau^i/\varepsilon^2\le t) \\
	&\,\le\, \frac{\varepsilon^{-2}}{t}\int_0^t\sup_{x\in[0,\varepsilon\tau]}\bigl|\PP_s^{\varepsilon,i}\mathcal{Q}f(x,i)-\mathcal{Q}f(x,i)\bigr|\D s +2\varepsilon^{-2}\max_{ i\in[n]}\lVert \mathcal{Q}f(\cdot,i)\rVert_\infty \max_{ i\in[n]}\Prob(\uptau^i/\varepsilon^2\le t)\,.
	\end{align*}	Observe that $\mathcal{Q}f(\cdot,i)\in\mathcal{C}_{u,b}^2(\R^d,\R)$.
	Hence, similarly as above we conclude  $$\lim_{t \to 0}	\bigl\|\PP_s^{\varepsilon,i}\mathcal{Q}f(x,i)-\mathcal{Q}f(x,i)\bigr\|_\infty\,=\,0\,,$$
	which  implies $$
	\left\lVert\frac{1}{t}\int_0^t\bigl(\PP^\varepsilon_s\mathcal{Q}f(\cdot,i)-\mathcal{Q}f(\cdot,i)\bigr)\D s\right\rVert_\infty\,=\,0\,.
	$$
	This concludes the proof.
\end{proof}

Following \cite{Freidlin-1964} (see also \cite[Lemma 3.4.1]{Bensoussan-Lions-Papanicolaou-Book-1978}), for $\varepsilon>0$  let $\bar \X^\varepsilon(x,i;t)\df \varepsilon^{-1}\X^\varepsilon(\varepsilon x,i;\varepsilon^2 t)$, $t\ge0$. Clearly, $\{\bar \X^\varepsilon(x,i;t)\}_{t\ge0}$ satisfies \begin{align*}\D \bar \X^\varepsilon(x,i;t)&\,=\, \bigl(\bb\bigl(\bar \X^\varepsilon(x,i;t),\Lambda(i;t)\bigr)+\varepsilon \cc\bigl(\bar \X^\varepsilon(x,i;t),\Lambda(i;t)\bigr)\bigr)\,\D t\\
&\ \ \ \ \ +\upsigma\bigl(\bar \X^\varepsilon(x,i;t),\Lambda(i;t)\bigr)\,\D \B(t)\\
\bar \X^\varepsilon(x,i;0)&\,=\, x\in\R^d\\
\Lambda(i;0)&\,=\,i\in[n]\,.\end{align*} Let also $\{\bar \X^0(x,i;t)\}_{t\ge0}$ be a solution to \begin{align*}\D \bar \X^0(x,i;t)&\,=\, \bb\bigl(\bar \X^0(x,i;t),\Lambda(i;t)\bigr)\,\D t+\upsigma\bigl(\bar \X^0(x,i;t),\Lambda(i;t)\bigr)\,\D \B(t)\\
\bar \X^0(x,i;0)&\,=\,x\in\R^d\\
\Lambda(i;0)&\,=\,i\in[n]\,.\end{align*} The processes $\{\bar \X^\varepsilon(x,i;t)\}_{t\ge0}$, $\varepsilon\ge0$, possess the same structural properties as $\{ \X^\varepsilon(x,i;t)\}_{t\ge0}$, $\varepsilon>0$, mentioned above.
Denote by $\bar \p^\varepsilon(t,(x,i),\D y\times\{j\})=\Prob\bigl((\bar \X^\varepsilon(x,i;t),\Lambda(i;t))\in\D y\times\{j\}\bigr)$,   $\{\mathcal{\bar P}^\varepsilon_t\}_{t\ge0}$ and  $(\bar{\mathcal{A}}^\varepsilon,\mathcal{D}_{\bar{\mathcal{A}}^\varepsilon})$ the corresponding transition kernel, operator semigroup and  $\mathcal{B}_b$-infinitesimal generator,   respectively.  Next, observe that due to $\tau$-periodicity of the coefficients, $\{ \bar\X^\varepsilon(x+ k_{\tau},i;t)\}_{t\ge0}$ and $\{ \bar\X^\varepsilon(x,i;t)+ k_{\tau}\}_{t\ge0}$   are indistinguishable. In particular, $$ \bar{\p}^\varepsilon(t,(x+ k_{\tau},i),B\times\{j\})\,=\, \bar{\p}^\varepsilon(t,(x,i),B- k_{\tau}\times\{j\})$$ for all  $B\in\mathfrak{B}(\R^d),$
which implies that $\{\bar{\PP}^\varepsilon_t\}_{t\ge0}$ preserves the class of $ \tau$-periodic functions in $\mathcal{B}_b(\R^d\times[n],\R)$. Thus,  according to \cite[Proposition 3.8.3]{Kolokoltsov-Book-2011}   the projection of  $\{(\bar \X^{\varepsilon}(x,i;t),\Lambda(i;t))\}_{t\ge0}$ with respect to $\Pi_{\tau}(x,i)$ on  $\mathbb{T}^d_{\tau}\times[n]$, denoted by   $\{(\bar \X^{\varepsilon,\tau}(x,i;t),\Lambda(i;t))\}_{t\ge0}$, is a Markov process on $(\mathbb{T}^d_{\tau}\times[n],\mathfrak{B}(\mathbb{T}^d_{\tau}\times[n]))$ with transition kernel  given by
\begin{equation}\label{ES2.6}
\bar \p^{\varepsilon,\tau}(t,(x,i),B\times\{j\})\,=\,\bar \p^\varepsilon\bigl(t,(z_x,i),\Pi_{\tau}^{-1}(B\times\{j\})\bigr)\end{equation} for  $(x,i)\in \mathbb{T}^d_{\tau}\times[n]$, $B\times\{j\} \in \mathfrak{B}(\mathbb{T}^d_{\tau}\times[n])$ and $(z_x,i) \in \Pi_{\tau}^{-1}(\{(x,i)\})$.
In particular, 
$\{(\bar \X^{\varepsilon,\tau}(x,i;t),\Lambda(i;t))\}_{t\ge0}$ is  a $\mathcal{C}_b$-Feller  process.
This, together with    \cite[Theorem 3.1]{Meyn-Tweedie-AdvAP-II-1993}, implies that $\{(\bar \X^{\varepsilon,\tau}(x,i;t),\Lambda(i;t))\}_{t\ge0}$ admits at least one invariant probability measure. Assuming additionally  \textbf{(A2)} (see below), in what follows we show that $\{(\bar \X^{\varepsilon,\tau}(x,i;t),\Lambda(i;t))\}_{t\ge0}$ 
admits one, and only one, invariant measure, and the corresponding marginals  converge as $t\to\infty$  to the  invariant measure in the total variation norm (with exponential rate). Assume

\medskip

\begin{description}
	\item[(A2)]   the matrix $ \upsigma(x,i)\upsigma(x,i)^{\mathrm{T}}$ is uniformly elliptic, that is, there is $c>0$ such that $$\xi^\mathrm{T}\upsigma(x,i)\upsigma(x,i)^\mathrm{T}\xi \,\ge\, c|\xi|^2\qquad \forall\, (x,i,\xi)\in\R^d\times [n]\times \R^d\,,$$ 
	and the Markov chain  $\{\Lambda(i;t)\}_{t\ge0}$ is irreducible, that is,
	for any $i,j \in [n]$, $i\neq j$, there are $o\in\N$ and $k_0, \dots, k_o \in [n]$ with $k_0=i$, $k_o=j$ and $k_l \neq k_{l+1}$ for $l=0,\dots,o-1$,   such that 
	$\mathrm{q}_{k_{l}k_{l+1}}>0$  for all $l=0,\dots,o-1$.
\end{description}	

\medskip

\noindent According to \cite[Theorem 2.1]{Lazic-Sadric-2022}, \textbf{(A2)} (together with \textbf{(A1)}) implies that  the process \linebreak 
 $\{(\bar \X^{\varepsilon,\tau}(x,i;t),\Lambda(i;t))\}_{t\ge0}$  is  irreducible with respect to the measure $$\uppsi(B\times\{i\})\,\df\, \mathrm{Leb}\bigl(\mathrm{pr}_{\R^d}\circ\Pi_{\tau}^{-1}(B\times\{i\})\bigr)\,,\qquad B\times\{i\}\in\mathfrak{B}(\mathbb{T}_\tau^d\times[n])\,,$$
 in the sense of \cite{Down-Meyn-Tweedie-1995}, that is, $\uppsi(B\times\{j\})>0$ implies that $\int_0^\infty \bar{\p}^{\varepsilon\tau}(t,(x,i),B\times\{j\})\,\D t>0$ for all $(x,i)\in\mathbb{T}_\tau^d\times[n]$. Here, $\mathrm{Leb}$ stands for the Lebesgue measure on $\mathfrak{B}(\R^d)$ and $\mathrm{pr}_{\R^d}:\R^d\times[n]\to\R^d$ is the projection mapping onto the first coordinate. This automatically entails that $\{(\bar \X^{\varepsilon,\tau}(x,i;t),\Lambda(i;t))\}_{t\ge0}$ admits one, and only one, invariant probability measure $\uppi^\varepsilon$. Namely, 
 according to \cite[Theorem 2.3]{Tweedie-1994} every irreducible Markov process is either  transient or recurrent. Due to the fact that   $\{(\bar \X^{\varepsilon,\tau}(x,i;t),\Lambda(i;t))\}_{t\ge0}$  admits at least one invariant probability measure  it clearly cannot be transient. The assertion then follows  from \cite[Theorem 2.6]{Tweedie-1994} which  states that every recurrent Markov process admits a unique (up to constant multiplies) invariant measure.

\begin{proposition} \label{P2.2}Under \textbf{(A1)}-\textbf{(A2)},
	there are $\gamma>0$ and $\varGamma>0$, such that $$\sup_{(x,i)\in\mathbb{T}_\tau^d\times[n]}\|\bar{\p}^{\varepsilon,\tau}(t,(x,i),\D y\times\{j\})-\uppi(\D y\times\{j\})\|_{\mathrm{TV}}\,\le\,\varGamma\E^{-\gamma t}\qquad \forall\,(\varepsilon,t)\in[0,\infty)\times[0,\infty)\,,$$ where $\lVert\cdot\rVert_{\mathrm{TV}}$ denotes the total variation norm on the space of
	signed measures on $\mathfrak{B}(\mathbb{T}_\tau^d\times[n])$.
\end{proposition}
\begin{proof} First, \cite[Theorems 5.1 and 7.1]{Tweedie-1994} together with the $\mathcal{C}_b$-Feller property and irreducibility  of $\{(\bar \X^{\varepsilon,\tau}(x,i;t),\Lambda(i;t))\}_{t\ge0}$ with respect to $\uppsi$  imply that $\mathbb{T}_\tau^d\times[n]$ is a petite set for  $\{(\bar \X^{\varepsilon,\tau}(x,i;t),$\linebreak$\Lambda(i;t))\}_{t\ge0}$ (see \cite{Tweedie-1994} for  the definition of petite sets). Next, from \cite[Theorem 4.2]{Meyn-Tweedie-AdvAP-III-1993} (with $c=d=1$, $C=\mathbb{T}_\tau^d\times[n]$, $f(x,i)=V(x,i)\equiv1$, and $\mathcal{A}V(x,i)\equiv0$), \cite[Theorem 2.1]{Lazic-Sadric-2022}  (which also implies that  $\sum_{k=1}^\infty\bar{\p}^{\varepsilon,\tau}(k,(x,i),B\times\{j\})>0$ for all $(x,i)\in\mathbb{T}_\tau^d\times[n]$ whenever $\uppsi(B\times\{j\})>0$) and \cite[Proposition 6.1]{Meyn-Tweedie-AdvAP-II-1993} we see that $\{(\bar \X^{\varepsilon,\tau}(x,i;t),\Lambda(i;t))\}_{t\ge0}$is aperiodic in the sense of \cite{Down-Meyn-Tweedie-1995}. The desired result now follows from \cite[Theorem 5.2]{Down-Meyn-Tweedie-1995} by taking $c=b=1$, $C=\mathbb{T}_\tau^d\times[n]$, $\tilde V(x,i)\equiv1$, and $\tilde{\mathcal{A}}\tilde V(x,i)\equiv0$.
\end{proof}

\medskip

\noindent As a consequence of \Cref{P2.2} we conclude  that for any $f\in\mathcal{B}_b(\mathbb{T}_\tau^d\times[n],\R)$ satisfying $\uppi^\varepsilon(f)=0$ it holds that \begin{equation}\label{ES2.7}\|{\bar{\mathcal{P}}}_t^{\varepsilon,\tau}f(\cdot,i)\|_\infty\,\le\,\varGamma\|f(\cdot,i)\|_\infty\E^{-\gamma t} \qquad \forall\,  t\ge0\,,\end{equation}
where $\{\bar{\mathcal{P}}_t^{\varepsilon,\tau}\}_{t\ge0}$
stands for the semigroup of $\{(\bar \X^{\varepsilon,\tau}(x,i;t),\Lambda(i;t))\}_{t\ge0}$. Denote the corresponding $\mathcal{B}_b$-infinitesimal generator by
$(\bar{\mathcal{A}}^{\varepsilon,\tau},\mathcal{D}_{\bar{\mathcal{A}}^{\varepsilon,\tau}})$.

\begin{proposition}\label{P2.3}
Assume \textbf{(A1)}-\textbf{(A2)}. Then,

\smallskip

\begin{itemize}
	\item [(i)] $\{\bar{\mathcal{P}}_t^{\varepsilon,\tau}\}_{t\ge0}$ is a $\mathcal{C}_\infty$-semigroup, that is,    $\bar{\mathcal{P}}^{\varepsilon,\tau}_t(\mathcal{C}(\mathbb{T}_\tau^d\times[n],\R))\subseteq\mathcal{C}(\mathbb{T}_\tau^d\times[n],\R)$ for all $t\ge0$ and $\lim_{t \to 0}\sup_{(x,i)\in\mathbb{T}_\tau^d\times[n] }|\bar{\mathcal{P}}^{\varepsilon,\tau}_tf(x,i)-f(x,i)|=0$ for all $f\in\mathcal{C}(\mathbb{T}_\tau^d\times[n],\R)$ 
	
	\smallskip
	
	\item[(ii)] $\mathcal{C}^2(\mathbb{T}_\tau^d\times[n],\R)\subseteq\mathcal{D}_{\bar{\mathcal{A}}^{\varepsilon,\tau}}$ and $\bar{\mathcal{A}}^{\varepsilon,\tau}f(x,i)=\bar{\mathcal{L}}^\varepsilon f(z_x,i)$ for $(x,i)\in\mathbb{T}_\tau^d\times[n]$, $(z_x,i)\in\Pi_\tau^{-1}(\{(x,i)\})$ and $f\in\mathcal{C}^2(\mathbb{T}_\tau^d\times[n],\R)$, where $$\bar{\mathcal{L}}^\varepsilon\,=\,2^{-1}\mathrm{Tr}\bigl(\upsigma(\cdot,i)\upsigma(\cdot,i)^\mathrm{T}\,\nabla\nabla^\mathrm{T}\bigr)+
	\bigl(\bb(\cdot,i)^{\mathrm{T}}+\varepsilon\cc(\cdot,i)^{\mathrm{T}}\bigr)\nabla+\mathcal{Q}$$
	
	\smallskip
	
	\item[(iii)] $\mathcal{C}^2(\mathbb{T}_\tau^d\times[n],\R)$ is an operator core for  $\bar{\mathcal{A}}^{\varepsilon,\tau}$, that is, $\bar{\mathcal{A}}^{\varepsilon,\tau}$ is the only extension of \linebreak $\bar{\mathcal{A}}^{\varepsilon,\tau}|_{\mathcal{C}^2(\mathbb{T}_\tau^d\times[n],\R)}$ on $\mathcal{D}_{\bar{\mathcal{A}}^{\varepsilon,\tau}}$.
	\end{itemize}
\end{proposition}
\begin{proof}
	\begin{itemize}
		\item [(i)] The $\mathcal{C}_\infty$-semigroup property follows directly from the definition of $\{\bar{\mathcal{P}}_t^{\varepsilon,\tau}\}_{t\ge0}$ (see \cref{ES2.6}) and \cite[Theorem 17.6]{Kallenberg-Book-1997}.
		
		\smallskip
		
		\item[(ii)] Clearly, for any  $f\in \mathcal{C}^2(\R^d\times[n],\R)$ the process $$\left\{f\bigl(\bar{\X}^\varepsilon(x,i;t),\Lambda(i;t)\bigr)-f(x,i)-\int_0^t\bar{\mathcal{L}}^\varepsilon f\bigl(\bar{\X}^\varepsilon(x,i;s),\Lambda(i;s)\bigr)\D s\right\}_{t\ge0}$$ is a $\mathbb{P}$-local martingale. By analogous reasoning as in \Cref{P2.1} we see that $\mathcal{C}_{u,b}^2(\R^{d}\times[n],\R)\subseteq \mathcal{D}_{\bar{\mathcal{A}}^\varepsilon}$ and $\bar{\A}^\varepsilon|_{\mathcal{C}_{u,b}^2(\R^{d}\times[n],\R)}=\bar{\mathcal{L}}^\varepsilon.$ This and the definition of $\{\bar{\mathcal{P}}_t^{\varepsilon,\tau}\}_{t\ge0}$ then give  \begin{align*}&\lim_{t \to 0}\sup_{(x,i)\in\mathbb{T}_\tau^d\times[n] }\left|\frac{\bar{\mathcal{P}}_t^{\varepsilon,\tau}f(x,i)-f(x,i)}{t} -\bar{\mathcal{L}}^\varepsilon f(z_x,i)\right|\\&\,=\,\lim_{t \to 0}\sup_{(x,i)\in\R^d\times[n] }\left|\frac{\bar{\mathcal{P}}_t^{\varepsilon}f(x,i)-f(x,i)}{t} -\bar{\mathcal{L}}^\varepsilon f(x,i)\right|\,=\,0\qquad \forall\, f\in\mathcal{C}^2(\mathbb{T}^{d}_{ \tau}\times[n],\R)\,.\end{align*}
		
		\smallskip
		
		\item[(iii)] According to \cref{ES2.6}  and \cite[Proposition 17.9]{Kallenberg-Book-1997} it suffices to prove that $\bar{\mathcal{P}}_t^{\varepsilon} (\mathcal{C}^2(\mathbb{T}_\tau^d\times[n],\R))\subseteq \mathcal{C}^2(\mathbb{T}_\tau^d\times[n],\R)$ for all $t\ge0$. In \cite[Corrollary 2.32]{Yin-Zhu-Book-2010} it has been shown that for each $i\in[n]$ and $t\ge0$, $x\mapsto\bar{\X}^\varepsilon(x,i;t)$ is of class $\mathcal{C}^2$ in $\mathrm{L}^2(\Prob)$ sense. The assertion now follows by analogous reasoning as in \cite[Theorem 5.5]{Friedman-Book-1975}.
		\end{itemize}
	\end{proof}

Finally, we show the following weak convergence result.

\begin{proposition}\label{P2.4}Under \textbf{(A1)}-\textbf{(A2)},
	$$\uppi^\varepsilon(\D x)\,\xRightarrow[\varepsilon\to0]{({\rm w})}\,\uppi^0(\D x)\,,$$ where $\ \xRightarrow[]{({\rm w})}$
	stands for the weak convergence of probability measures.
\end{proposition}
\begin{proof}
	The result will follow if we show that
	\begin{equation}\label{ES2.8}\lim_{\varepsilon\to0}\sup_{(x,i)\in\mathbb{T}_\tau^d\times[n] }|\bar{\mathcal{P}}^{\varepsilon,\tau}_tf(x,i)-\bar{\mathcal{P}}^{0,\tau}_tf(x,i)|\,=\,0\qquad \forall\,(t,f)\in [0,\infty)\times\mathcal{C}(\mathbb{T}_\tau^d\times[n],\R)\,.\end{equation}
	Indeed,	since $\mathbb{T}_\tau^d\times[n]$ is compact the family of probability measures $\{\uppi^\varepsilon\}_{\varepsilon\ge0}$ is  tight. Hence, for any sequence $\{\varepsilon_i\}_{i\in\N}\subset[0,\infty)$ converging to $0$ there is a further subsequence $\{\varepsilon_{i_j}\}_{j\in\N}$ such that $\{\uppi^{\varepsilon_{i_j}}\}_{j\in\N}$ converges weakly to some probability measure $\bar\uppi^0$. Take $f\in\mathcal{C}(\mathbb{T}_\tau^d\times[n],\R)$, and fix $t\ge0$ and $\epsilon>0$. From \cref{ES2.8} we have that there is $\varepsilon_0>0$ such that $$\sup_{(x,i)\in\mathbb{T}_\tau^d\times[n] }|\bar{\mathcal{P}}^{\varepsilon,\tau}_tf(x,i)-\bar{\mathcal{P}}^{0,\tau}_tf(x,i)|\,\le\,\epsilon\qquad \forall\,\varepsilon\in[0,\varepsilon_0]\,.$$ 
	We now have that
	\begin{align*}
	|\bar\uppi^{0}(f)-\bar\uppi^0\bigl(\bar{\mathcal{P}}_t^{0,\tau}f\bigr)|&\,=\,\lim_{j\to\infty}|\uppi^{\varepsilon_{i_j}}(f)-\bar\uppi^0\bigl(\bar{\mathcal{P}}_t^{0,\tau}f\bigr)|\\&\,=\,\lim_{j\to\infty}|\uppi^{\varepsilon_{i_j}}\bigl(\bar{\mathcal{P}}_t^{\varepsilon_{i_j},\tau}f\bigr)-\bar\uppi^0\bigl(\bar{\mathcal{P}}_t^{0,\tau}f\bigr)|\\&\le\,
	\limsup_{j\to\infty}|\uppi^{\varepsilon_{i_j}}\bigl(\bar{\mathcal{P}}_t^{\varepsilon_{i_j},\tau}f\bigr)-\uppi^{\varepsilon_{i_j}}\bigl(\bar{\mathcal{P}}_t^{0,\tau}f\bigr)| +\lim_{j\to\infty}|\uppi^{\varepsilon_{i_j}}\bigl(\bar{\mathcal{P}}_t^{0,\tau}f\bigr)-\bar\uppi^0\bigl(\bar{\mathcal{P}}_t^{0,\tau}f\bigr)|\\
	&\,\le\,\epsilon\,,
	\end{align*}
	which implies that  $\bar\uppi^0(\D x)$ is an invariant probability measure for $\process{\bar{\X}^{0,\tau}}$. Thus,  $\bar\uppi^0(\D x)=\uppi^0(\D x)$, which concludes the result.
	
	To this end let us show \cref{ES2.8}. According to \cite[Theorem 17.25]{Kallenberg-Book-1997} and \Cref{P2.3} (iii), \cref{ES2.8} will follow if we show that 
	$$ \lim_{\varepsilon \to 0}\sup_{(x,i)\in\mathbb{T}_\tau^d\times[n] }|\bar{\mathcal{A}}^{\varepsilon,\tau}f(x,i)-\bar{\mathcal{A}}^{0,\tau}f(x,i)|\,=\,0\qquad \forall\, f\in \mathcal{C}^2(\mathbb{T}_\tau^d\times[n],\R)\,.$$ This is a direct consequence of \Cref{P2.3} (ii).
\end{proof}

	Observe that if $\cc\equiv0$, then
	 $\{\bar{\X}^\varepsilon(x,i;t)\}_{t\ge0}=\{\bar{\X}^0(x,i;t)\}_{t\ge0}$ and  $\{\X^\varepsilon(x,i;t)\}_{t\ge0}=$\linebreak$\{\varepsilon\bar{\X}^0(x/\varepsilon,i;t/\varepsilon^2)\}_{t\ge0}$ for all $\varepsilon>0$.
Hence, \Cref{P2.4} trivially holds and in \textbf{(A1)} it suffices to assume that $\bb\in \mathcal{C}^1(\mathbb{T}_\tau^d\times[n],\R^d)$.

%%%%%%%%%%%%%%%%%%%%%%%%%%%%%%%%%%%%%%%%
%%%%%%%%%%%%%%%%%%%%%%%%%%%%%%%%%%%%%%%%%

\section{CLT for the process $\process{\X}$}
\label{S3}

In this section, we prove a functional CLT for the process $\process{\X^\varepsilon}$.

\begin{theorem}\label{T3.1} Assume \textbf{(A1)}-\textbf{(A2)}. Then
	$$\{  \X^\varepsilon(x,i;t)-\varepsilon^{-1}\uppi^0(\bb)t\}_{t\ge0}\,\xRightarrow[\varepsilon\to0]{({\rm d})}\,\{\mathrm{W}^{\mathsf{a},\mathsf{b}}(x;t)\}_{t\ge0}\,,$$ where
	 $\{\mathrm{W}^{\mathsf{a},\mathsf{b}}(x;t)\}_{t\ge0}$ is a  $d$-dimensional zero-drift  Brownian motion   determined by covariance matrix and drift vector \begin{equation}\label{ES3.1} \mathsf{a}\,=\,\uppi^0\bigl((\Id_d-\DD\beta)\,\upsigma\upsigma^\mathrm{T}\,(\Id_d-\DD\beta)^{\mathrm{T}}\bigr)\qquad \text{and}\qquad \mathsf{b}\,=\,\uppi^{0}\bigl((\Id_n-\DD\beta)\cc\bigr)\,, \end{equation} respectively. Here,  $\ \xRightarrow[]{({\rm d})}$ stands for the convergence in the space
	 of continuous functions endowed with the locally uniform topology (see \cite[Chapter VI]{Jacod-Shiryaev-2003} for
	 details). 
	\end{theorem}

\medskip

The function $\beta:\R^d\times[n]\to\R^d$ appearing in \cref{ES3.1} is given in the following lemma.

\begin{lemma}\label{L3.2}Assume \textbf{(A1)}-\textbf{(A2)}. There is one, and only one (up to a constant),  $\beta\in\mathcal{C}^2(\mathbb{T}_\tau^d\times[n],\R^d)$ solving \begin{equation}\label{ES3.2}\bar{\mathcal{L}}^0\beta(x,i)\,=\,\mathrm{b}(x,i)-\uppi^0(\mathrm{b})\,.\end{equation}
		\end{lemma}
\begin{proof}
	First,  observe that  \cref{ES2.7} implies $$\lVert\bar{\mathcal{P}}^0_t\bigl(\mathrm{b}-\uppi^0(\mathrm{b})\bigr)(\cdot,i)\rVert_\infty\,\le\,\varGamma\lVert\mathrm{b}(\cdot,i)-\uppi(\mathrm{b})\rVert_\infty\E^{-\gamma t}\,.$$ Hence, $$\beta(x,i)\,\df\,-\int_0^\infty\bar{\mathcal{P}}^0_t\bigl(\mathrm{b}-\uppi(\mathrm{b})\bigr)(x,i)\,\D t$$ is well defined,  $\beta\in\mathcal{C}(\mathbb{T}_\tau^d\times[n],\R^d)$, satisfies $\uppi^0(\beta)=0$ and solves \begin{equation}\label{ES3.3} \bar{\mathcal{A}}^0\beta(x,i)\,=\,\mathrm{b}(x,i)-\uppi^0(\mathrm{b})\,.\end{equation} Actually, $\beta(x,i)$ is a unique $\tau$-periodic  bounded solution to \cref{ES3.3} satisfying $\uppi^0(\beta)=0$.
	Namely, if $\bar\beta(x,i)$ is an another such solution, then  \cite[Proposition 4.1.7]{Ethier-Kurtz-Book-1986} gives that $$\beta(x,i)-\bar{\beta}(x,i)\,=\,\bar\PP^0_t\beta(x,i)-\bar\PP^0_t\bar{\beta}(x,i)\qquad \forall\, t\ge0\,.$$
	By letting now $t\to\infty$, \cref{ES2.7} proves the claim.
	
	Further, let $\rho>0$ be arbitrary. Obviously $\beta(x,i)$ solves the Dirichlet problem \begin{equation}\begin{aligned}\label{ES3.4}
	\bar{\mathcal{A}}^0\tilde{\beta}(x,i)&\,=\,\mathrm{b}(x,i)-\uppi^0(\mathrm{b})\qquad\text{in}\quad \mathscr{B}_\rho(0)\times[n]\\\
	\tilde\beta(x,i)&\,=\,\beta(x,i)\qquad\text{on}\quad \partial\mathscr{B}_\rho(0)\times[n]\,.
	\end{aligned}\end{equation}  
Actually, if $\tilde\beta\in\mathcal{D}_{\bar{\mathcal{A}}^0}$ is an another  solution to \cref{ES3.4}, then it necessarily coincides with $\beta(x,i)$ on $\overline{\mathscr{B}}_\rho(0)$. Indeed, let $\tilde{\beta}(x,i)$ be an another such solution.  
Then, $\beta(x,i)-\tilde{\beta}(x,i)$ solves 
\begin{equation}\begin{aligned}\label{ES3.5}
\bar{\mathcal{A}}^0\hat{\beta}(x,i)&\,=\,0\qquad\text{in}\quad \mathscr{B}_\rho(0)\times[n] \\
\hat\beta(x,i)&\,=\,0\qquad\text{on}\quad \partial\mathscr{B}_\rho(0)\times[n]\,.
\end{aligned}\end{equation}  Hence, it suffices to prove that any solution $\hat\beta\in\mathcal{D}_{\bar{\mathcal{A}}^0}$ to \cref{ES3.5} must vanish on $\overline{\mathscr{B}}_\rho(0)$. Let $\hat\beta(x,i)$ be such solution. 	Define $\uptau^{x,i}\df\inf\{t\ge0:\bar\X^0(x,i;t)\notin \mathscr{B}_\rho(0)\}.$ According to \cite[Lemma 3.2]{Zhu-Yin-Baran-2015}, $$\mathbb{E}[\uptau^{x,i}]\,<\,\infty\qquad\forall\ (x,i)\in \mathscr{B}_\rho(0)\times[n]\,.$$ From \cite[Proposition 4.1.7]{Ethier-Kurtz-Book-1986} and optimal stopping theorem we now conclude $$  \hat{\beta}(x,i)\,=\,\mathbb{E}\left[\hat{\beta}\bigl(\bar\X^0(x,i;\uptau^{x,i}),\Lambda(i;\uptau^{x,i}\bigr)\right]-\mathbb{E}\left[\int_0^{\uptau^{x,i}}\bar{\mathcal{A}}^0\hat{\beta}\bigl(\bar\X^0(x,i;s),\Lambda(x,i;s)\bigr)\D s\right]\,=\,0$$
for all $(x,i)\in \mathscr{B}_\rho(0)\times[n],$ which proves the assertion.

Finally, we shown that the Dirichlet problem 
\begin{equation}\label{ES3.6}
\begin{aligned}
\bar{\mathcal{L}}^0\check{\beta}(x,i)&\,=\,\mathrm{b}(x,i)-\uppi^0(\mathrm{b})\qquad\text{in}\quad \mathscr{B}_\rho(0)\times[n]\\
\check\beta(x,i)&\,=\,\beta(x,i)\qquad\text{on}\quad \partial\mathscr{B}_\rho(0)\times[n]\,.
\end{aligned} 
\end{equation}
admits a unique solution $\check{\beta}\in\mathcal{C}(\overline{\mathscr{B}}_\rho(0)\times[n],\R^d)\cap\mathcal{C}^2(\mathscr{B}_\rho(0)\times[n],\R^d).$ If this is the case,  by extending $\check{\beta}(x,i)$ to a $\mathcal{C}^2_{u,b}(\R^d\times[n],\R^d)$-function, which we again denote by $\check{\beta}(x,i)$, and using the facts that $\mathcal{C}^2_b(\R^d\times[n],\R^d)\subseteq\mathcal{D}_{\bar{\mathcal{A}}^0}$ and $\bar{\mathcal{A}}^0|_{\mathcal{C}^2_{u,b}(\R^d\times[n],\R^d)}=\bar{\mathcal{L}}^0$ (see \Cref{P2.1}), we conclude that  $\beta(x,i)=\check{\beta}(x,i)$ on $\overline{\mathscr{B}}_\rho(0)\times[n]$. This finishes the proof.

To this end, let us show that \cref{ES3.6} admits a unique solution $\check{\beta}\in\mathcal{C}(\overline{\mathscr{B}}_\rho(0)\times[n],\R^d)\cap\mathcal{C}^2(\mathscr{B}_\rho(0)\times[n],\R^d).$
Consider the following problems 	\begin{equation}\label{ES3.7}
\begin{aligned}
\bar{\mathcal{L}}_i^{0} \check{\beta}^{0}(x,i)+\mathrm{q}_{ii}\check{\beta}^{0}(x,i)&\,=\,\mathrm{b}(x,i)-\uppi^0(\mathrm{b})\,,\qquad x\in\mathscr{B}_\rho(0)\,,\\
\check{\beta}^{0}(x,i)&\,=\, \beta(x,i)\,,\qquad x\in\partial\mathscr{B}_\rho(0)\,,
\end{aligned}
\end{equation}
and, for $k\ge1$,  	\begin{equation}\label{ES3.8}
\begin{aligned}
\bar{\mathcal{L}}_i^{0}\check{\beta}^{k}(x,i)+\mathrm{q}_{ii}\check{\beta}^{k}(x,i)+\sum_{j\neq i}\mathrm{q}_{ij}\check{\beta}^{k}(x,j)& \,=\,\mathrm{b}(x,i)-\uppi^0(\mathrm{b})\,,\qquad x\in\mathscr{B}_\rho(0)\,,\\
\check{\beta}^{k}(x,i)&\,=\, \beta(x,i)\,,\qquad x\in\partial\mathscr{B}_\rho(0)\,.
\end{aligned}
\end{equation}
Here, 
$$\bar{\mathcal{L}}^0_i\,=\,2^{-1}\mathrm{Tr}\bigl(\upsigma(\cdot,i)\upsigma(\cdot,i)^\mathrm{T}\,\nabla\nabla^\mathrm{T}\bigr)+
\bb(\cdot,i)^{\mathrm{T}}\nabla\,.$$
According to \cite[Theorem 6.13]{Gilbarg-Trudinger-Book-2001}, \cref{ES3.7,ES3.8} admit unique solutions
$\check{\beta}^{k}(x,i),$ $k\ge0$,  of class $\mathcal{C}(\overline{\mathscr{B}}_\rho(0)\times[n],\R^d)\cap\mathcal{C}^2(\mathscr{B}_\rho(0)\times[n],\R^d)$. Furthermore,  
\cite[Theorem 3.47]{Pardoux-Rascanu-Book-2014} implies that
\begin{align*}\check{\beta}^{0}(x,i)\,=\,\mathbb{E}\Big[&\beta\bigl(\bar\X^{0,i}(x;\bar{\uptau}^{x,i}),i\bigr)\E^{\mathrm{q}_{ii}\bar{\uptau}^{x,i}}-\int_0^{\bar{\uptau}^{x,i}}\bigl(\bb\bigl(\bar\X^{0,i}(x;s),i\bigr)-\uppi^0(\bb)\bigr)\E^{\mathrm{q}_{ii}s}\D s \Big]\end{align*}
and \begin{equation}\label{ES3.9}
\begin{aligned}\check{\beta}^{k}(x,i)\,=\,\mathbb{E}&\Big[\beta\bigl(\bar\X^{0,i}(x;\bar{\uptau}^{x,i}),i\bigr)\E^{\mathrm{q}_{ii}\bar{\uptau}^{x,i}}\\&-\int_0^{\bar{\uptau}^{x,i}}\left(\bb\bigl(\bar\X^{0,i}(x;s),i\bigr)-\uppi(\bb)-\sum_{j\neq i}\mathrm{q}_{ij}\check{\beta}^{k-1}\bigl(\bar\X^{0,i}(x;s),j\bigr)\right) \E^{\mathrm{q}_{ii}s}\D s \Big]\,,\end{aligned}\end{equation}
where 
  $\{\bar \X^{0,i}(x;t)\}_{t\ge0}$ is a solution to
  \begin{equation} \begin{aligned}\label{ES3.10}\D \bar \X^{0,i}(x;t)&\,=\, \bb\bigl(\bar \X^{0,i}(x;t),i\bigr)\,\D t+\upsigma\bigl(\bar \X^{0,i}(x;t),i\bigr)\,\D \B(t)\\
\bar \X^{0,i}(x;0)&\,=\,x\in\R^d\,,\end{aligned}\end{equation} and $\bar{\uptau}^{x,i}=\inf\{t\ge0:\bar\X^{0,i}(x;t)\notin \overline{\mathscr{B}}_\rho(0)\}$.
Due to \textbf{(A2)} and \cite[Remark 3.48]{Pardoux-Rascanu-Book-2014},
\begin{equation}\label{ES3.11}\max_{ i\in[n]}\sup_{x\in\overline{\mathscr{B}}_\rho(0)}\mathbb{E}[\bar{\uptau}^{x,i}]\,<\,\infty\,.\end{equation}
Observe that 
\begin{align*}&| \check{\beta}^{k+1}(x,i) -\check{\beta}^{k}(x,i)|\\&\,\le\,\mathbb{E}\left[\int_0^{\bar{\uptau}^{x,i}}\sum_{j\neq i}\mathrm{q}_{ij}\left| \check{\beta}^{k}(x,i)\bigl(\bar\X^{0,i}(x;s),j\bigr) -\check{\beta}^{k-1}(x,i)\bigl(\bar\X^{0,i}(x;s),j\bigr)\right| \E^{\mathrm{q}_{ii}s}\D s \right] \,.
\end{align*}
Thus, \begin{align*}\max_{ i\in[n]}\sup_{x\in\overline{\mathscr{B}}_\rho(0)}| \check{\beta}^{k+1}(x,i) -\check{\beta}^{k}(x,i)|\,\le\,&\max_{ i\in[n]}\sup_{x\in\overline{\mathscr{B}}_\rho(0)}| \check{\beta}^{k}(x,i) -\check{\beta}^{k-1}(x,i)|\\&   \max_{ i\in[n]}\sup_{x\in\overline{\mathscr{B}}_\rho(0)}\left|1-\mathbb{E}\bigl[\E^{\mathrm{q}_{ii}\bar{\uptau}^{x,i}}\bigr]\right|\,.\end{align*} From \cref{ES3.11} we have that $$\max_{ i\in[n]}\sup_{x\in\overline{\mathscr{B}}_\rho(0)}\bigl\lVert1-\mathbb{E}\bigl[\E^{\mathrm{q}_{ii}\bar{\uptau}^{x,i}}\bigr]\bigr\rVert\,<\,1\,.$$
Hence, the sequence $\{\check{\beta}^{k}\}_{k\ge1}$ converges uniformly on $\overline{\mathscr{B}}_\rho(0)$ to  $\check{\beta}\in\mathcal{C}(\overline{\mathscr{D}}\times[n],\R^d)$. In particular, $$\sup_{i\in[n],\, k\ge0}\sup_{x\in\overline{\mathscr{B}}_\rho(0)}| \check{\beta}^{k}(x,i)|\,<\,\infty\,.$$ From \cite[Theorem 9.11]{Gilbarg-Trudinger-Book-2001} we now conclude that for any $p\ge1$ and any  $0<\rho'<\rho$, $$\sup_{i\in[n],\, k\ge0}\lVert \check{\beta}^{k}\rVert_{\mathcal{W}^{2,p}(\mathscr{B}_{\rho'}(0)\times[n],\R^d)}\,<\,\infty\,.$$ Here,  $\mathcal{W}^{2,p}(\mathscr{B}_{\rho'}(0)\times[n],\R^d)$ denotes the space of functions $h:\mathscr{B}_{\rho'}(0)\times[n]\to\R^d$ such that $h(\cdot,i)\in\mathcal{W}^{2,p}(\mathscr{B}_{\rho'}(0),\R^d)$ (the  $(2,p)$-Sobolev  space).
In particular, $\check\beta\in\mathcal{W}^{2,p}(\mathscr{B}_{\rho'}(0)\times[n],\R^d)$. This, together with \cite[Corollary 7.11]{Gilbarg-Trudinger-Book-2001},  gives that
$\check\beta\in\mathcal{C}^1(\mathscr{B}_{\rho}(0)\times[n],\R^d)$. Now,  \cite[Theorem 6.13]{Gilbarg-Trudinger-Book-2001} implies that
\begin{align*}
\bar{\mathcal{L}}_i^{0} \mathring{\beta}(x,i)+\mathrm{q}_{ii}\mathring{\beta}(x,i)+\sum_{j\neq i}\mathrm{q}_{ij}\check\beta(x,j) &\,=\,\bb(x,i)-\uppi^0(\bb)\,,\qquad x\in\mathscr{B}_{\rho}(0)\,,\\
\mathring{\beta}(x,i)&\,=\, \beta(x,i)\,,\qquad x\in\partial\mathscr{B}_{\rho}(0)\,,
\end{align*}
admits a unique $\mathcal{C}(\overline{\mathscr{B}}_{\rho}(0)\times[n],\R)\cap\mathcal{C}^2(\mathscr{B}_{\rho}(0)\times[n],\R)$ solution.
Furthermore, since $\check{\beta}^{k}(x,i)$ satisfies \cref{ES3.9}, it follows that 
\begin{align*}\check{\beta}(x,i)\,=\,\mathbb{E}&\Big[\beta\bigl(\bar\X^{0,i}(x;\bar{\uptau}^{x,i}),i\bigr)\E^{\mathrm{q}_{ii}\bar{\uptau}^{x,i}}\\&-\int_0^{\bar{\uptau}^{x,i}}\left(\bigl(\bb\bigl(\bar\X^{0,i}(x;s),i\bigr)-\uppi(\bb)\bigr)-\sum_{j\neq i}\mathrm{q}_{ij}\check{\beta}(\bigl(\bar\X^{0,i}(x;s),j\bigr)\right) \E^{\mathrm{q}_{ii}s}\D s \Big]\,,\end{align*} which together with \cite[Theorems 3.47 and  3.49]{Pardoux-Rascanu-Book-2014} gives that $\check\beta(x,i)=\mathring{\beta}(x,i)$. 
\end{proof}

We now  prove  \Cref{T3.1}.

\begin{proof}[Proof of \Cref{T3.1}] Let $\beta\in\mathcal{C}^2(\mathbb{T}_\tau^d\times[n],\R^d)$  be a solution to \cref{ES3.3}. By	It\^{o}'s formula we have that
	\begin{align*}
	&\beta\bigl(\bar\X^\varepsilon(x,i;t),\Lambda(i;t)\bigr)\\&\,=\,\beta(x,i)+\int_0^t\bar{\mathcal{L}}^\varepsilon\beta\bigl(\bar\X^\varepsilon(x,i;s),\Lambda(i;s)\bigr)\D s +\int_0^t\bigl((\mathrm{D}\beta)\upsigma\bigr)\bigl(\bar\X^\varepsilon(x,i;s),\Lambda(i;s)\bigr)\D\B(s)\\
	&\,=\,\beta(x,i)+\int_0^t\bar{\mathcal{L}}^0\beta\bigl(\bar\X^\varepsilon(x,i;s),\Lambda(i;s)\bigr)\D s+\varepsilon\int_0^t\bigl((\mathrm{D}\beta)\cc\bigr)\bigl(\bar\X^\varepsilon(x,i;s),\Lambda(i;s)\bigr)\D s\\
	&\ \ \ \ \  +\int_0^t\bigl((\mathrm{D}\beta)\upsigma\bigr)\bigl(\bar\X^\varepsilon(x,i;s),\Lambda(i;s)\bigr)\D\B(s)\,.
	\end{align*}
	Thus,
		\begin{align*}&\X^{\varepsilon}(x,i;t)-\varepsilon^{-1}\uppi^0(\bb)t-x\\
	&\,=\, \varepsilon\bar{\X}^{\varepsilon}(x/\varepsilon;t/\varepsilon^2)-\varepsilon^{-1}\uppi^0(\bb)t-x\\
	 &\,=\,\int_0^{t/\varepsilon^2}\left(\varepsilon\bigl(\bb-\uppi^0(\bb)\bigr)\bigr(\bar\X^\varepsilon(x/\varepsilon,i;s),\Lambda(i;s)\bigr) +\varepsilon^2\cc\bigl(\bar\X^\varepsilon(x/\varepsilon,i;s),\Lambda(i;s)\bigr)\right)\D s\\
	&\ \ \ \ \ +\varepsilon\int_0^{t/\varepsilon^2}\upsigma\bigl(\bar\X^\varepsilon(x/\varepsilon,i;s),\Lambda(i;s)\bigr)\D \B(s)\\
	&\,=\, \varepsilon\beta\bigl(\bar\X^\varepsilon(x/\varepsilon,i;t/\varepsilon^2),\Lambda(i;t/\varepsilon^2)\bigr)-\varepsilon\beta(x/\varepsilon,i)\\
	&\  \ \ \ \ + \varepsilon^2\int_0^{t/\varepsilon^2}\bigl((\Id_d-\mathrm{D}\beta)c\bigr)\bigl(\bar\X^\varepsilon(x/\varepsilon,i;s),\Lambda(i;s)\bigr)\D s\\
	&\ \ \ \ \ + \varepsilon\int_0^{t/\varepsilon^2}\bigl((\Id_d-\mathrm{D}\beta)\upsigma\bigr)\bigl(\bar\X^\varepsilon(x/\varepsilon,i;s),\Lambda(i;s)\bigr)\D \B(s)\,,
			\end{align*}
that is, \begin{align*}&\X^{\varepsilon}(x,i;t)-\varepsilon^{-1}\uppi^0(\bb)t-x-\varepsilon\beta\bigl(\bar\X^\varepsilon(x/\varepsilon,i;t/\varepsilon^2),\Lambda(i;t/\varepsilon^2)\bigr)+\varepsilon\beta(x/\varepsilon,i)\\
	&\,=\, \varepsilon^2\int_0^{t/\varepsilon^2}\bigl((\Id_d-\mathrm{D}\beta)c\bigr)\bigl(\bar\X^\varepsilon(x/\varepsilon,i;s),\Lambda(i;s)\bigr)\D s\\
	&\ \ \ \ \ + \varepsilon\int_0^{t/\varepsilon^2}\bigl((\Id_d-\mathrm{D}\beta)\upsigma\bigr)\bigl(\bar\X^\varepsilon(x/\varepsilon,i;s),\Lambda(i;s)\bigr)\D \B(s)\,.
	\end{align*}
	From this we see
	
	\smallskip
	
	\begin{itemize}
		\item [(i)] $\{\X^{\varepsilon}(x,i;t)-\varepsilon^{-1}\uppi^0(\bb)t-x-\varepsilon\beta(\bar\X^\varepsilon(x/\varepsilon,i;t/\varepsilon^2),\Lambda(i;t/\varepsilon^2))+\varepsilon\beta(x/\varepsilon,i)\}_{t\ge0}$  converges in law if, and only if, $\{\X^{\varepsilon}(x,i;t)-\varepsilon^{-1}\uppi^0(\bb)t-x\}_{t\ge0}$ converges, and if this is the case the limit is the same. 
		
		\smallskip
		
		\item [(ii)] $\{\X^{\varepsilon}(x,i;t)-\varepsilon^{-1}\uppi^0(\bb)t-x-\varepsilon\beta(\bar\X^\varepsilon(x/\varepsilon,i;t/\varepsilon^2),\Lambda(i;t/\varepsilon^2))+\varepsilon\beta(x/\varepsilon,i)\}_{t\ge0}$ is a semimartingale with bounded variation and predictable quadratic covariation parts $$\left\{\varepsilon^2\int_0^{t/\varepsilon^2}\bigl((\Id_d-\mathrm{D}\beta)c\bigr)\bigl(\bar\X^\varepsilon(x/\varepsilon,i;s),\Lambda(i;s)\bigr)\D s\right\}_{t\ge0}$$and $$\left\{\varepsilon^2\int_0^{t/\varepsilon^2}\bigl((\Id_d-\mathrm{D}\beta)\upsigma\upsigma^{\mathrm{T}}(\Id_d-\mathrm{D}\beta))^{\mathrm{T}}\bigr)\bigl(\bar\X^\varepsilon(x/\varepsilon,i;s),\Lambda(i;s)\bigr)\D s\right\}_{t\ge0}\,,$$ respectively.
	\end{itemize}

\smallskip

\noindent	From  \cite[Theorem VI.3.21]{Jacod-Shiryaev-2003} we see that both  these processes are tight. Consequently, \cite[Theorem VI.4.18]{Jacod-Shiryaev-2003}  implies tightness of $\{\X^{\varepsilon}(x,i;t)-\varepsilon^{-1}\uppi^0(\bb)t-x-\varepsilon\beta(\bar\X^\varepsilon(x/\varepsilon,i;t/\varepsilon^2),\Lambda(i;t/\varepsilon^2))+\varepsilon\beta(x/\varepsilon,i)\}_{t\ge0}$. To this end, it remains to prove finite-dimensional convergence in law of\linebreak  $\{\X^{\varepsilon}(x,i;t)-\varepsilon^{-1}\uppi^0(\bb)t-x-\varepsilon\beta(\bar\X^\varepsilon(x/\varepsilon,i;t/\varepsilon^2),\Lambda(i;t/\varepsilon^2))+\varepsilon\beta(x/\varepsilon,i)\}_{t\ge0}$ to $\{\mathrm{W}^{\mathsf{a},\mathsf{b}}(x;t)\}_{t\ge0}$.
	According to  \cite[Theorem VIII.2.4]{Jacod-Shiryaev-2003} this will hold if
	$$ \varepsilon^2\int_0^{\varepsilon^{-2}t}\bigl((\Id_d-\DD\beta)\cc\bigr)\bigl(\bar\X^\varepsilon(x/\varepsilon,i;s),\Lambda(i;s)\bigr)\D s\,\xrightarrow[\varepsilon\to0]{\Prob}\,\mathsf{b}\,t
	\qquad \forall\ t\ge0\,,$$ and 
		 $$ \varepsilon^2\int_0^{\varepsilon^{-2}t}\bigl((\Id_d-\DD\beta)\upsigma\upsigma^\mathrm{T}(\Id_d-\DD\beta)^{\mathrm{T}}\bigr)\bigl(\bar\X^\varepsilon(x/\varepsilon,i;s),\Lambda(i;s)\bigr)\D s\,\xrightarrow[\varepsilon\to0]{\Prob}\,\mathsf{a}\,t
 \qquad \forall\ t\ge0\,,$$ where $\ \xrightarrow[]{\Prob}$ denotes the  convergence in probability.
 Due to $\tau$-periodicity, we have that
 \begin{align*}&\varepsilon^2\int_0^{\varepsilon^{-2}t}\bigl((\Id_d-\DD\beta)\cc\bigr)\bigl(\bar\X^\varepsilon(x/\varepsilon,i;s),\Lambda(i;s)\bigr)\D s\\&\,=\,\varepsilon^2\int_0^{\varepsilon^{-2}t}\bigl((\Id_d-\DD\beta)\cc\bigr)\bigl(\bar\X^{\varepsilon,\tau}(\Pi_{\tau}(x/\varepsilon,i);s),\Lambda(i;s)\bigr)\D s\qquad \forall\,t\ge0\,,\end{align*} and an analogous relation holds for the predictable quadratic covariation part. Further,
  \begin{align*}
 &\varepsilon^4\mathbb{E}\Bigg[\left(\int_0^{\varepsilon^{-2}t}\bigl((\Id_d-\DD\beta)\cc-\uppi^\varepsilon\bigl((\Id_d-\DD\beta)\cc\bigr)\bigr)\bigl(\bar\X^{\varepsilon,\tau}(\Pi_{\tau}(x/\varepsilon,i);s),\Lambda(i;s)\bigr)\,\D s\right)^{\mathrm{T}}\\&\hspace{1.6cm}\left(\int_0^{\varepsilon^{-2}t}\bigl((\Id_d-\DD\beta)\cc-\uppi^\varepsilon\bigl((\Id_d-\DD\beta)\cc\bigr)\bigr)\bigl(\bar\X^{\varepsilon,\tau}(\Pi_{\tau}(x/\varepsilon,i);s),\Lambda(i;s)\bigr)\,\D s\right)\Bigg]\\
 &\,=\,2\varepsilon^4 \int_0^{\varepsilon^{-2}t}\int_0^{s}\mathbb{E}\bigg[\Big(\bigl((\Id_d-\DD\beta)\cc-\uppi^\varepsilon\bigl((\Id_d-\DD\beta)\cc\bigr)\bigr)\bigl(\bar\X^{\varepsilon,\tau}(\Pi_{\tau}(x/\varepsilon,i);s),\Lambda(i;s)\bigr)\Big)^{\mathrm{T}}\\
 &\hspace{3.6cm}\Big(\bigl((\Id_d-\DD\beta)\cc-\uppi^\varepsilon\bigl((\Id_d-\DD\beta)\cc\bigr)\bigr)\bigl(\bar\X^{\varepsilon,\tau}(\Pi_{\tau}(x/\varepsilon,i);v),\Lambda(i;v)\bigr)\Big)\bigg]\D v\,\D s\\
 &\,=\,2\varepsilon^4 \int_0^{\varepsilon^{-2}t}\int_0^{s}\mathbb{E}\bigg[\Big(\bar{\mathcal{P}}^{\varepsilon,\tau}_{s-v}\bigl((\Id_d-\DD\beta)\cc-\uppi^\varepsilon\bigl((\Id_d-\DD\beta)\cc\bigr)\bigr)\bigl(\bar\X^{\varepsilon,\tau}(\Pi_{\tau}(x/\varepsilon,i);v),\Lambda(i;v)\bigr)\Big)^{\mathrm{T}}\\&\hspace{3.5cm}\Big(\bigl((\Id_d-\DD\beta)\cc-\uppi^\varepsilon\bigl((\Id_d-\DD\beta)\cc\bigr)\bigr)\bigl(\bar\X^{\varepsilon,\tau}(\Pi_{\tau}(x/\varepsilon,i);v),\Lambda(i;v)\bigr)\Big)\bigg]\D v\,\D s\\
 &\,\le\,8\varepsilon^4\varGamma\max_{i\in[n]}\|((\Id_d-\DD\beta)\cc)(\cdot,i)\|^2_\infty \int_0^{\varepsilon^{-2}t}\int_0^{s}\E^{-\gamma(s-u)}\D u\,\D s\\
 &\,=\,\frac{8\varepsilon^4\varGamma\max_{i\in[n]}\|((\Id_d-\DD\beta)\cc)(\cdot,i)\|^2_\infty}{\gamma^2}\bigl(\varepsilon^{-2}t+\E^{-\gamma\varepsilon^{-2}t}-1\bigr)\,,
 \end{align*}
 where in the third step we employed  \cref{ES2.7}.  Hence 
 \begin{align*}
 &\varepsilon^2\Bigg(\mathbb{E}\Bigg[\left(\int_0^{\varepsilon^{-2}t}\bigl((\Id_d-\DD\beta)\cc-\mathsf{b}\bigr)\bigl(\bar\X^{\varepsilon}(x/\varepsilon,i;s),\Lambda(i;s)\bigr)\,\D s\right)^{\mathrm{T}}\\&\hspace{1.6cm}\left(\int_0^{\varepsilon^{-2}t}\bigl((\Id_d-\DD\beta)\cc-\mathsf{b}\bigr)\bigl(\bar\X^{\varepsilon}(x/\varepsilon,i;s),\Lambda(i;s)\bigr)\,\D s\right)\Bigg]\Bigg)^{1/2}\\
& \,\le\, \varepsilon^2\Bigg(\mathbb{E}\Bigg[\left(\int_0^{\varepsilon^{-2}t}\bigl((\Id_d-\DD\beta)\cc-\uppi^\varepsilon\bigl((\Id_d-\DD\beta)\cc\bigr)\bigr)\bigl(\bar\X^{\varepsilon}(x/\varepsilon,i;s),\Lambda(i;s)\bigr)\,\D s\right)^{\mathrm{T}}\\&\hspace{1.6cm}\left(\int_0^{\varepsilon^{-2}t}\bigl((\Id_d-\DD\beta)\cc-\uppi^\varepsilon\bigl((\Id_d-\DD\beta)\cc\bigr)\bigr)\bigl(\bar\X^{\varepsilon}(x/\varepsilon,i;s),\Lambda(i;s)\bigr)\,\D s\right)\Bigg]\Bigg)^{1/2}\\
&\ \ \ \ \ +|\uppi^\varepsilon\bigl((\Id_d-\DD\beta)\cc\bigr)-\mathsf{b}|\,t\,.
 \end{align*}
 Analogous estimate holds for  the predictable quadratic covariation part. By letting
$\varepsilon\to0$, the result follows from \Cref{P2.4}.
 	\end{proof}

In the following proposition we show  that the Brownian motion $\{\mathrm{W}^{\mathsf{a},\mathsf{b}}(x;t)\}_{t\ge0}$ appearing in \Cref{T3.1} is nondegenerate.
\begin{proposition}\label{P3.3} Assume  \textbf{(A1)}-\textbf{(A2)}. Then the matrix $\mathsf{a}$ is positive definite.
	\end{proposition}
\begin{proof}
Observe first that from  \textbf{(A1)}-\textbf{(A2)} it immediately follows that there are $\overline{c}\ge\underline{c}>0$ such that 
$$\overline{c}\,\xi^\mathrm{T}\uppi^0\bigl((\Id_d-\mathrm{D}\beta)(\Id_d-\mathrm{D}\beta)^\mathrm{T}\bigr)\xi\,\ge\,\xi^\mathrm{T}\mathsf{a}\,\xi\,\ge\underline{c}\,\xi^\mathrm{T}\uppi^0\bigl((\Id_d-\mathrm{D}\beta)(\Id_d-\mathrm{D}\beta)^\mathrm{T}\bigr)\xi\qquad \forall\,\xi\in\R^d\,.$$
Hence, it suffices to show that   $\xi^\mathrm{T}\uppi^0((\Id_d-\mathrm{D}\beta)(\Id_d-\mathrm{D}\beta)^\mathrm{T})\xi=0$ if, and only if, $\xi=0$. We have 
\begin{align*}
&\xi^\mathrm{T}\uppi^0\bigl((\Id_d-\mathrm{D}\beta)(\Id_d-\mathrm{D}\beta)^\mathrm{T}\bigr)\xi\,=\, \int_{\mathbb{T}_\tau^d\times[n]} \sum_{l=1}^d\left(\sum_{k=1}^d \partial_l\left(\xi_k\bigl(x_k-\beta_k(x,i)\bigr)\right)\right)^2\uppi^0(\D x\times\{i\})\,.
\end{align*}
The left-hand side in the above relation vanishes if, and only if, 
\begin{equation}\label{ES3.12}\sum_{l=1}^d\left(\sum_{k=1}^d \partial_l\left(\xi_k\bigl(x_k-\beta_k(x,i)\bigr)\right)\right)^2\,=\,0\qquad \uppi\text{-a.e.}\end{equation}  
We claim that \cref{ES3.12} holds for all $(x,i)\in[0,\tau]\times[n]$.
Denote the transition kernel and semigroup of  $\{\bar \X^{0,i}(x;t)\}_{t\ge0}$ (recall \cref{ES3.10}) by $\bar\p^{0,i}(t,x, \D y )$ and $\{\bar{\mathcal{P}}^{0,i}_t\}_{t\ge0}$, respectively. 
From
the proof of \cite[Theorem 4.8]{Xi-Yin-Zhu-2019} we have that
\begin{equation}\begin{aligned}\label{ES3.13}
&\bar\p^0\bigl(t,(x,i), B \times \{ j \} \bigr) \,=\, 
\updelta_{ij}\, \bar\p^{0,i}(t,x,B)\, \E^{\mathrm{q}_{ii} t}\\ &+ (1-\updelta_{ij})\, \sum_{o=1}^\infty \underset{0<t_1<\ldots<t_o<t}{\idotsint} \underset{\substack{k_0,\ldots,k_o \in \mathbb{S} \\ k_l \neq k_{l+1}\\ k_0=i,\,  k_o=j}}{\sum} \int_{\R^d} \cdots \int_{\R^d} \bar\p^{0,k_0}(t_1,x,\D y_1) \E^{\mathrm{q}_{k_0k_0} t} \mathrm{q}_{k_0k_1} &\\
& \bar\p^{0,k_1}(t_2-t_1,y_1,\D y_2) \E^{\mathrm{q}_{k_1k_1} t} \mathrm{q}_{k_1k_2} \cdots \mathrm{q}_{k_{o-1}k_o} \bar\p^{0,k_o}(t-t_o,y_o,B) \E^{\mathrm{q}_{k_ok_o} t} \D t_1 \ldots \D t_o\,,
\end{aligned}\end{equation}
where $\updelta_{ij}$ is the Kronecker delta.
	Further, due to \cite[Theorems 7.3.3 and 7.3.4]{Durrett-Book-1996} there is a strictly positive function  $\bar p^{0,i}(t,x,y)$ on $(0,\infty)\times \R^d\times\R^d$, jointly continuous in $t$, $x$ and $y$, and twice continuously differentiable in $x$, satisfying \begin{equation*} \bar\PP^{0,i}_tf(x)\,=\,\int_{\R^d}\bar p^{0,i}(t,x,y)f(y)\,\D y\,,\qquad t>0,\, x\in \R^d,\, f\in C_b(\R^d,\R)\,.\end{equation*}  By employing dominated convergence theorem, the above relation holds also for $\mathbb{1}_O$, for any open set $O\subseteq \R^d$. Denote by $\mathcal{D}$ the class of all  $B\in\mathfrak{B}(\R^d)$  such that \begin{equation*} \bar\PP_t^{0,i}\mathbb{1}_B(x)\,=\,\int_{B}\bar p^{0,i}(t,x,y)\, \D y\,,\qquad t>0,\, x\in \R^d\,.\end{equation*} Clearly, $\mathcal{D}$ contains the $\pi$-system of open rectangles in $\R^d$, and forms a $\lambda$-system. Hence, by employing  Dynkin's $\pi$-$\lambda$ theorem we conclude that $\mathcal{D}=\mathfrak{B}(\R^d).$ Consequently, for any $t>0$, $x\in \R^d$ and $B\in\mathfrak{B}(\R^d)$ we have that
\begin{equation*} \bar\p^{0,i}(t,x,B)\,=\, \int_{B}\bar p^{0,i}(t,x,y)\, \D y\,.\end{equation*} 
From this and \cref{ES3.13} we see that $\bar\p^{0}(t,(x,i), \D y \times \{ j \} )$ is equivalent (in the absolute continuity sense) with $\mathrm{Leb}(\mathrm{pr}_{\R^d}(\D y\times{j}))$, which in turn implies that $\bar\p^{0,\tau}(t,(x,i), \D y \times \{ j \} )$ is equivalent with $\mathrm{Leb}(\mathrm{pr}_{\R^d}\circ\Pi_\tau^{-1}(\D y\times{j}))$.
From the invariance property of $\uppi^0$, that is, $$\int_{\mathbb{T}^{d}_\tau\times[n]}\bar{\p}^{0,\tau}(t,(x,i), \D y \times \{ j \} )\uppi^0(\D x\times\{i\})\,=\,\uppi^0(\D y\times\{j\})\qquad \forall\, t>0\,,$$ it now follows that $\uppi^0$ is  equivalent with $\mathrm{Leb}(\mathrm{pr}_{\R^d}\circ\Pi_\tau^{-1}(\D y\times{j}))$. Hence, since $\beta\in \mathcal{C}^2(\mathbb{T}_\tau^d\times[n],\R^d)$,  \cref{ES3.12} holds for all $(x,i)\in[0,\tau]\times[n]$. 
Assume now that $\xi\neq0$. Let $\xi_l$ be its nonvanishing coordinate coefficient. From \cref{ES3.12} we then have that $$\xi_lx_l\,=\,\xi_l\beta_l(x,i)-\sum_{k\neq l}\xi_k\bigl(x_k-\beta_k(x,i)\bigr)+f_l(x,i)\qquad\forall\,(x,i)\in[0,\tau]\times[n]$$ for some $f\in\mathcal{C}(\mathbb{T}_\tau^d\times[n],\R)$ which does not depend on the $l$-th coordinate in $x$.
However, the above relation cannot hold due to $\tau$-periodicity of $\beta(x,i)$, unless $\xi_l=0$. Thus, $\xi=0$.
\end{proof}

\section{Homogenization of  weakly coupled systems of linear   PDEs}\label{S4}

In this section, we prove the homogenization results for problems in \cref{ES1.2,ES1.3}.

	\subsection{The elliptic  problem in \cref{ES1.2}.} We first discuss the  system of elliptic boundary-valued problems in \cref{ES1.2}.

\begin{theorem}\label{T4.1}	In addition to \textbf{(A1)}-\textbf{(A2)},  assume
	\begin{itemize}
		\item[(a)] $\mathscr{D}$ is an open bounded subset of $\R^d$ satisfying the following:
		
		\smallskip
		
		\begin{itemize}
			\item [(i)] $\mathscr{D}=\left\{x: \mathscr{d}(x)<0 \right\}$ for some $\mathscr{d} \in \mathcal{C}_b^2(\R^d,\R)$
			
			\smallskip
			
			\item [(ii)] $\left| \nabla \mathscr{d}(x)\right| \ge \delta $ for all $x \in \partial\mathscr{D}$ and some $\delta>0$
		\end{itemize}

		\smallskip
		
		\item[(b)] $\uppi^0(\bb)=0$ (otherwise we replace $\bb(x,i)$ by $\bb(x,i)-\uppi^0(\bb)$ in \cref{ES1.2})
		
		\smallskip
		
		\item [(c)] $e\in\mathcal{C}(\mathbb{T}_\tau^d\times[n],\R)$ and satisfies $e(x,i)\le -\alpha$ for some $\alpha>0$
		
		\smallskip
		
		\item[(d)] $f\colon \mathscr{D}\to\R$ is H\"{o}lder continuous and $g\colon\partial\mathscr{D}\to\R$ is  continuous.
		%\smallskip		
%		\item [(e)] $\{x\in \partial \mathscr{D}:\Prob(\hat{\uptau}^{x}>0)=0 \}$ is a closed set and 
%		$$\Prob\left(\nabla \mathscr{d}\bigl(\mathrm{W}^{\mathsf{a}}(x;\hat{\uptau}^{x})\bigr)^\mathrm{T}\mathsf{a}\,\nabla \mathscr{d}\bigl(\mathrm{W}^{\mathsf{a}}(x;\hat{\uptau}^{x})\bigr) >0\right)\,=\,1\qquad \forall\, x\in\mathscr{D}\,,$$
%		where
		%$\hat{\uptau}^{x}\df\inf\{t\ge0:\mathrm{W}^{\mathsf{a}}(x;t)\notin \mathscr{D}\}$.
	\end{itemize}
	
\noindent	Then,  the problem in \cref{ES1.2} admits a unique  solution  $u^\varepsilon\in\mathcal{C}(\overline{\mathscr{D}}\times[n],\R)\cap\mathcal{C}^2(\mathscr{D}\times[n],\R)$ and $$ \lim_{\varepsilon \to 0} u^\varepsilon(x,i)\,=\, u(x)\qquad \forall\,(x,i)\in \mathscr{D}\times[n]\,,$$ where $u\in \mathcal{C}(\overline{\mathscr{D}},\R)\cap \mathcal{C}^2(\mathscr{D},\R)$ is a unique solution  to
\begin{equation}	\begin{aligned}\label{ES4.1}
	2^{-1}\mathrm{Tr}\bigl(\mathsf{a}\,\nabla\nabla^\mathrm{T}\bigr)u(x)+\mathsf{b}^\mathrm{T}\nabla u(x)+\uppi^0(e)\, u(x)+f(x)&\,=\, 0\,,\qquad x\in \mathscr{D}\,,\\
	u(x)&\,=\,g(x)\,,\qquad x\in \partial \mathscr{D}\,.
	\end{aligned}\end{equation}
\end{theorem}
\begin{proof} We first show existence, uniqueness and smoothness of $u^\varepsilon(x,i)$. We proceed analogously as in the proof of  \Cref{L3.2}.
Consider the following problems 	\begin{equation}\label{ES4.2}
\begin{aligned}
\mathcal{L}_i^{\varepsilon} u^{\varepsilon,0}(x,i)+\varepsilon^{-2}\mathrm{q}_{ii}u^{\varepsilon,0}(x,i)+
e(x/\varepsilon,i)\,u^{\varepsilon,0}(x,i)+f(x)&\,=\,0\,,\qquad x\in\mathscr{D}\,,\\
u^{\varepsilon,0}(x,i)&\,=\, g(x)\,,\qquad x\in\partial\mathscr{D}\,,
\end{aligned}
\end{equation}
and, for $k\ge1$,  	\begin{equation}\label{ES4.3}
 \begin{aligned}
 \mathcal{L}_i^{\varepsilon} u^{\varepsilon,k}(x,i)+\varepsilon^{-2}\mathrm{q}_{ii}u^{\varepsilon,k}(x,i)&+\varepsilon^{-2}\sum_{j\neq i}\mathrm{q}_{ij}u^{\varepsilon,k-1}(x,j) \\&+
 e(x/\varepsilon,i)\,u^{\varepsilon,k}(x,i)+f(x)\,=\,0\,,\qquad x\in\mathscr{D}\,,\\
 &\hspace{3.25cm}u^{\varepsilon,k}(x,i)\,=\, g(x)\,,\qquad x\in\partial\mathscr{D}\,.
 \end{aligned}
 \end{equation}
According to \cite[Theorem 6.13]{Gilbarg-Trudinger-Book-2001}, \cref{ES4.2,ES4.3} admit unique solutions
 $u^{\varepsilon,k}(x,i),$ $k\ge0$,  of class $\mathcal{C}(\overline{\mathscr{D}}\times[n],\R)\cap\mathcal{C}^2(\mathscr{D}\times[n],\R)$. Furthermore,  
\cite[Theorem 3.47]{Pardoux-Rascanu-Book-2014} implies that
\begin{align*}u^{\varepsilon,0}(x,i)\,=\,\mathbb{E}\Big[&g\bigl(\X^{\varepsilon,i}(x;\uptau^{\varepsilon,x,i})\bigr)\E^{\int_0^{\uptau^{\varepsilon,x,i}} (e(\X^{\varepsilon,i}(x;s)/\varepsilon,i)+\varepsilon^{-2}\mathrm{q}_{ii})\D s}\\&+\int_0^{\uptau^{\varepsilon,x,i}}f\bigl(\X^{\varepsilon,i}(x;s)\bigr)\E^{\int_0^{s} (e(\X^{\varepsilon,i}( x;v)/\varepsilon,i)+\varepsilon^{-2}\mathrm{q}_{ii})\D v}\D s \Big]\end{align*}
and \begin{equation}\label{ES4.4}
\begin{aligned}u^{\varepsilon,k}(x,i)\,=\,\mathbb{E}\Big[&g\bigl(\X^{\varepsilon,i}(x;\uptau^{\varepsilon,x,i})\bigr)\E^{\int_0^{\uptau^{\varepsilon,x,i}} (e(\X^{\varepsilon,i}(x;s)/\varepsilon,i)+\varepsilon^{-2}\mathrm{q}_{ii})\D s}\\&+\int_0^{\uptau^{\varepsilon,x,i}}\left(f\bigl(\X^{\varepsilon,i}(x;s)\bigr)+\varepsilon^{-2}\sum_{j\neq i}\mathrm{q}_{ij}u^{\varepsilon,k-1}\bigl(\X^{\varepsilon,i}(x;s),j\bigr)\right)\\&\ \ \ \ \E^{\int_0^{s} (e(\X^{\varepsilon,i}(x;v)/\varepsilon,i)+\varepsilon^{-2}\mathrm{q}_{ii})\D v}\D s \Big]\,,\end{aligned}\end{equation}
where $\uptau^{\varepsilon,x,i}=\inf\{t\ge0:\X^{\varepsilon,i}(x;t)\notin \overline{\mathscr{D}}\}.$ Due to \textbf{(A2)} and \cite[Remark 3.48]{Pardoux-Rascanu-Book-2014},
\begin{equation}\label{ES4.5}\max_{ i\in[n]}\sup_{x\in\overline{\mathscr{D}}}\mathbb{E}[\uptau^{\varepsilon,x,i}]\,<\,\infty\,.\end{equation}
Observe that 
\begin{align*}&| u^{\varepsilon,k+1}(x,i) -u^{\varepsilon,k}(x,i)|\\&\,\le\,\mathbb{E}\left[\int_0^{\uptau^{\varepsilon,x,i}}\varepsilon^{-2}\sum_{j\neq i}\mathrm{q}_{ij}\left| u^{\varepsilon,k}\bigl(\X^{\varepsilon,i}(x;s),j\bigr) -u^{\varepsilon,k-1}\bigl(\X^{\varepsilon,i}(x;s),j\bigr)\right| \E^{\int_0^{s} (e(\X^{\varepsilon,i}(x;v)/\varepsilon,i)+\varepsilon^{-2}\mathrm{q}_{ii})\D v}\D s \right] \,.
\end{align*}
Thus, \begin{align*}\max_{ i\in[n]}\sup_{x\in\overline{\mathscr{D}}}| u^{\varepsilon,k+1}(x,i) -u^{\varepsilon,k}(x,i)|\,\le\,&\max_{ i\in[n]}\sup_{x\in\overline{\mathscr{D}}}| u^{\varepsilon,k}(x,i) -u^{\varepsilon,k-1}(x,i)|\\&   \max_{ i\in[n]}\sup_{x\in\overline{\mathscr{D}}}\left|1-\mathbb{E}\bigl[\E^{\varepsilon^{-2}\mathrm{q}_{ii}\hat{\uptau}^{\varepsilon,x,i}}\bigr]\right|\,.\end{align*} From \cref{ES4.5} we have that $$\max_{ i\in[n]}\sup_{x\in\overline{\mathscr{D}}}\bigl\lVert1-\mathbb{E}\bigl[\E^{\varepsilon^{-2}\mathrm{q}_{ii}\uptau^{\varepsilon,x,i}}\bigr]\bigr\rVert\,<\,1\,.$$
Hence, the sequence $\{u^{\varepsilon,k}\}_{k\ge1}$ converges uniformly on $\overline{\mathscr{D}}$ to  $u^\varepsilon\in\mathcal{C}(\overline{\mathscr{D}}\times[n],\R)$. In particular, $$\sup_{i\in[n],\, k\ge0}\sup_{x\in\overline{\mathscr{D}}}| u^{\varepsilon,k}(x,i)|\,<\,\infty\,.$$ From \cite[Theorem 9.11]{Gilbarg-Trudinger-Book-2001} we now conclude that for any $p\ge1$ and any open set $\mathscr{D}'\subset \mathscr{D}$ with compact closure in $\mathscr{D}$, $$\sup_{i\in[n],\, k\ge0}\lVert u^{\varepsilon,k}\rVert_{\mathcal{W}^{2,p}(\mathscr{D}'\times[n],\R)}\,<\,\infty\,,$$ where  $\mathcal{W}^{2,p}(\mathscr{D}'\times[n],\R)$ denotes the space of functions $h:\mathscr{D}'\times[n]\to\R$ such that $h(\cdot,i)\in\mathcal{W}^{2,p}(\mathscr{D}',\R^d)$.
 In particular, $u^\varepsilon\in\mathcal{W}^{2,p}(\mathscr{D}'\times[n],\R)$. This, together with \cite[Corollary 7.11]{Gilbarg-Trudinger-Book-2001},  gives that
$u^{\varepsilon}\in\mathcal{C}^1(\mathscr{D}\times[n],\R)$. Now,  \cite[Theorem 6.13]{Gilbarg-Trudinger-Book-2001} implies that
\begin{align*}
\mathcal{L}_i^{\varepsilon} \bar{u}^{\varepsilon}(x,i)+\varepsilon^{-2}\mathrm{q}_{ii}\bar{u}^{\varepsilon}(x,i)&+\varepsilon^{-2}\sum_{j\neq i}\mathrm{q}_{ij}u^{\varepsilon}(x,j) \\&+
e(x/\varepsilon,i)\,\bar{u}^{\varepsilon}(x,i)+f(x)\,=\,0\,,\qquad x\in\mathscr{D}\,,\\
&\hspace{3.25cm}\bar{u}^{\varepsilon}(x,i)\,=\, g(x)\,,\qquad x\in\partial\mathscr{D}\,,
\end{align*}
admits a unique $\mathcal{C}(\overline{\mathscr{D}}\times[n],\R)\cap\mathcal{C}^2(\mathscr{D}\times[n],\R)$ solution.
Furthermore, since $u^{\varepsilon,k}(x,i)$ satisfies \cref{ES4.4}, it follows that 
\begin{align*}u^{\varepsilon}(x,i)\,=\,\mathbb{E}\Big[&g\bigl(\X^{\varepsilon,i}(x;\uptau^{\varepsilon,x,i})\bigr)\E^{\int_0^{\uptau^{\varepsilon,x,i}} (e(\X^{\varepsilon,i}(x;s)/\varepsilon,i)+\varepsilon^{-2}\mathrm{q}_{ii})\D s}\\&+\int_0^{\uptau^{\varepsilon,x,i}}\left(f\bigl(\X^{\varepsilon,i}(x;s)\bigr)+\varepsilon^{-2}\sum_{j\neq i}\mathrm{q}_{ij}u^{\varepsilon}\bigl(\X^{\varepsilon,i}(x;s),j\bigr)\right) \E^{\int_0^{s} (e(\X^{\varepsilon,i}(x;v)/\varepsilon,i)+\varepsilon^{-2}\mathrm{q}_{ii})\D v}\D s \Big]\,,\end{align*} which together with \cite[Theorems 3.47 and  3.49]{Pardoux-Rascanu-Book-2014} gives that $u^{\varepsilon}(x,i)=\bar{u}^{\varepsilon}(x,i)$.

We next show that $\lim_{\varepsilon\to0} u^\varepsilon(x,i)=u(x)$ for all $(x,i)\in\mathscr{D}\times[n]$, where $u\in\mathcal{C}(\overline{\mathscr{D}},\R)\cap \mathcal{C}^2(\mathscr{D},\R)$ is a unique  solution to \cref{ES4.1} (see \cite[Theorem 6.13]{Gilbarg-Trudinger-Book-2001}). 	We follow the approach from \cite[Theorem 3.4.5]{Bensoussan-Lions-Papanicolaou-Book-1978}. First,  \cite[Theorem 5.1]{Friedman-Book-1975} and \cite[Theorem 3.3]{Zhu-Yin-Baran-2015} (see also \cite[eq. (4.7)]{Freidlin-Book-1985}) imply that 
$$u(x)\,=\,\mathbb{E}\Big[g\bigl(\mathrm{W}^{\mathsf{a},\mathsf{b}}(x;\hat{\uptau}^{x})\bigr)\E^{ \uppi(e)\hat{\uptau}^{x}}+\int_0^{\hat{\uptau}^{x}}f\bigl(\mathrm{W}^{\mathsf{a},\mathsf{b}}(x;s)\bigr) \E^{\uppi^0(e)s}\D s \Big]$$ 
and
\begin{align*}u^{\varepsilon}(x,i)\,=\,\mathbb{E}\Big[&g\bigl(\X^\varepsilon(x,i;\hat{\uptau}^{\varepsilon,x,i})\bigr)\E^{\int_0^{\hat{\uptau}^{\varepsilon,x,i}} e(\X^\varepsilon(x,i;s)/\varepsilon,\Lambda(i;s/\varepsilon^2))\D s}\\&+\int_0^{\hat{\uptau}^{\varepsilon,x,i}}f\bigl(\X^\varepsilon(x,i;s)\bigr) \E^{\int_0^{s} e(\X^\varepsilon(x,i;v)/\varepsilon,\Lambda(i;v/\varepsilon^{-2}))\D v}\D s \Big]\,,\end{align*} where $\hat{\uptau}^{x}=\inf\{t\ge0:\mathrm{W}^{\mathsf{a},\mathsf{b}}(x;t)\notin \mathscr{D}\}$ and $\hat{\uptau}^{\varepsilon,x,i}=\inf\{t\ge0:\X^\varepsilon(x,i;t)\notin \mathscr{D}\}.$ Due to \textbf{(A2)} and \cite[Lemma 3.2]{Zhu-Yin-Baran-2015},
$$\mathbb{E}[\hat{\uptau}^{x}]\,<\,\infty\quad\text{and}\quad\max_{ i\in[n]}\mathbb{E}[\hat{\uptau}^{\varepsilon,x,i}]\,<\,\infty\qquad\forall\, x\in\mathscr{D}\,.$$
 Define 
\begin{equation*}
\zeta^\varepsilon(x,i;t)\,\df\,  \int_0^{t}e\bigl( \X^\varepsilon(x,i;s)/\varepsilon,\Lambda(i;s/\varepsilon^2)\bigr)\,\D s\,.
\end{equation*}
Analogously as in the proof of \Cref{T3.1}, we conclude that 
\begin{equation}\label{ES4.6}
\zeta^\varepsilon(x,i;t)  \,\xrightarrow[\varepsilon\to0]{{\rm L}^2(\Prob)}\, \uppi^0\left(e\right)t\qquad\forall\ t\ge0\,,
\end{equation} where $\ \xrightarrow[]{{\rm L}^p(\Prob)}$ stands for the convergence in ${\rm L}^p(\Prob)$, $p\ge1$.
Set $\zeta(x;t)\df \uppi^0\left(e\right) t$.
 This, together with the fact that the process $\{\zeta^\varepsilon(x,i;t)\}_{t\ge0}$ is tight (due to \cite[Theorem VI.3.21]{Jacod-Shiryaev-2003}), implies that $\{\zeta^\varepsilon(x,i;t)\}_{t\ge0} \xRightarrow[\varepsilon\to0]{({\rm d})} \{\zeta(x;t)\}_{t\ge0}$. Since $\{\zeta(x;t)\}_{t\ge0}$ is a constant in the space $\mathcal{C}([0,\infty),\R)$, using \Cref{T3.1} and \cite[Theorem 3.9]{Billingsley-Book-1999} we conclude (recall that $\uppi^0(\bb)=0$)
\begin{equation}\label{ES4.7}
\{(\X^\varepsilon, \zeta^\varepsilon)(x,i;t)\}_{t\ge0}\, \xRightarrow[\varepsilon\to0]{({\rm d})}\, \{(\mathrm{W}^{\mathsf{a},\mathsf{b}}, \zeta)(x;t)\}_{t\ge0}\,.
\end{equation}
We next endow the space $ \mathcal{C}([0,\infty),\R^d)\times \mathcal{C}([0,\infty),\R)$  with the Borel $\sigma$-algebra generated by the sets $\left\{ (\mathrm{W}^{\mathsf{a},\mathsf{b}}, \xi)(x;\cdot) \in  \mathcal{C}([0,\infty),\R^d)\times \mathcal{C}([0,\infty),\R) \colon (\mathrm{W}^{\mathsf{a},\mathsf{b}}, \zeta)(x;s) \in B \right\}$ where $s \in [0,\infty)$ and $ B \in \mathfrak{B}(\R^{d+1})$. The processes $\{(\X^\varepsilon, \zeta^\varepsilon)(x,i;t)\}_{t\ge0}$ and $\{(\mathrm{W}^{\mathsf{a},\mathsf{b}}, \zeta)(x;t)\}_{t\ge0}$ introduce on $\mathcal{C}([0,\infty),\R^d)\times \mathcal{C}([0,\infty),\R)$  probability measures $\upmu_{x,i}^\varepsilon$ and $\upmu_x$, respectively. Observe that  $\upmu_{x,i}^\varepsilon\xRightarrow[\varepsilon\to0]{({\rm w})}\upmu_x$, where
   $\, \xRightarrow[]{({\rm w})}$
stands for the weak convergence of probability measures.
We next define  $\mathrm{F}: \mathcal{C}([0,\infty),\R^d)\times \mathcal{C}([0,\infty),\R) \to \R\cup\{\infty\}$ with
\begin{equation*}
\mathrm{F}(Y,\eta)\,=\,\begin{cases}
g\bigl(Y\bigl(\uptau (Y)\bigr)\bigr)\,\E^{\eta(\uptau (Y))}+\int_{0}^{\uptau(Y)}f\bigl(Y(t)\bigr)\,\E^{\eta(t)}\,\D t\,, &\uptau(Y)<\infty\ \ \text{and}\ \ \eta(t)\le -\alpha t\ \ \forall t\ge 0\,,

\smallskip

\\

\smallskip

\int_{0}^{\infty}f\bigl(Y(t)\bigr)\,\E^{\eta(t)}\,\D t\,, & \uptau(Y)=\infty\ \ \text{and}\ \ \eta(t)\le -\alpha t\ \ \forall t\ge 0\,,\\

\infty \,, & \text{otherwise}\,,
\end{cases}
\end{equation*}
where $\uptau(Y)\df\inf \{t\ge 0\colon Y(t) \notin \mathscr{D}\}$.  
Clearly,
\begin{equation*}
u^\varepsilon(x,i)\,=\,\mathbb{E} \left[\mathrm{F}\left(\{(\X^\varepsilon, \zeta^\varepsilon)(x,i;t)\}_{t\ge0}\right)\right]\qquad \text{and}\qquad u(x)\,=\, \mathbb{E} \left[\mathrm{F}\left(\{(\mathrm{W}^{\mathsf{a},\mathsf{b}}, \zeta)(x;t)\}_{t\ge0}\right)\right]\,.
\end{equation*}
The function $\mathrm{F}(Y,\eta)$ has the following properties:

\smallskip

\begin{itemize}
	\item[(i)] it is measurable and bounded a.s. with respect to $\upmu_{x,i}^\varepsilon$ and $\upmu_x$
	
	\smallskip
	
	\item[(ii)] it is continuous a.s. with respect to $\upmu_x$.
\end{itemize}

\smallskip

\noindent This, together with \cref{ES4.7}, gives 
\begin{equation*}
\lim\limits_{\varepsilon \to 0}u^\varepsilon(x,i)\,=\,u(x)\qquad \forall (x,i)\in\mathscr{D}\times[n]\,,
\end{equation*}
which concludes the proof.

To this end, let us  verify (i) and (ii).	To see that  (i) holds, note that if $\eta(t)\le -\alpha t$ for all $t\ge 0$, then
\begin{align*}
\left|\mathrm{F}(Y,\eta)\right|&\,\le\, \left\| g \right\|_\infty+\left\| f \right\|_\infty \int_{0}^{\infty}\E^{-\alpha t}\,\D t\,=\,\left\| g \right\|_\infty+ \frac{\left\| f \right\|_\infty}{\alpha}<\infty\,.
\end{align*}
Due to the definition of processes $\{\zeta^\varepsilon(x,i;t)\}_{t\ge0}$ and  $\{\zeta(x,t)\}_{t\ge0}$, and the fact that   $e(x,i)\le -\alpha$, property (i) follows.
To see property (ii), we need to check that  if $\{Y_k\}_{k\in\N}$ converges  to $Y$ uniformly on compact intervals, then 
\begin{equation}\label{ES4.8}
\lim_{k\to\infty}\mathrm{F}(Y_k,\eta)\,=\,\mathrm{F}(Y,\eta)
\end{equation} for $\eta(t)=\zeta(x,t)$. Recall that $\zeta(x,t)\le-\alpha t$.
The relation in \cref{ES4.8} will follow from the proof of \cite[Lemma 3.4.3]{Bensoussan-Lions-Papanicolaou-Book-1978} where it has been shown that if $\{Y_k\}_{k\in\N}$ converges  to $Y$ uniformly on compact intervals, then $\lim_{k \to \infty}\uptau(Y_k)=\uptau(Y)$. Namely, from this fact, and the dominated convergence theorem, we immediately see that  $$\lim_{k \to \infty}\int_{0}^{\uptau(Y_k)}f\bigl(Y_k(t)\bigr)\,\E^{\eta(t)}\D t \,=\, \int_{0}^{\uptau(Y)}f\bigl(Y(t)\bigr)\,\E^{\eta(t)}\D t\,.$$
Finally, we need to prove that
\begin{equation*}\lim_{k\to\infty}g\bigl(Y_k\bigl(\uptau (Y_k)\bigr)\bigr)\,\E^{\eta(\uptau (Y_k))} \mathbb{1}_{\{\uptau(Y_k)<\infty\}}\,=\,\begin{cases}
g\bigl(Y\bigl(\uptau (Y)\bigr)\bigr)\,\E^{\eta(\uptau (Y))}\,, &\uptau(Y)<\infty\,,

\smallskip

\\

\smallskip

0\,, &\uptau(Y)=\infty\,.\\

\end{cases}
\end{equation*}
If $\uptau(Y)=\infty$, then $\lim_{k \to \infty}\uptau(Y_k)=\infty$ and since $g(x)$ is bounded and $\eta(t)\le -\alpha t$ for all $t\ge0$, the assertion follows. If $\uptau(Y)<\infty$, then there exist $T>0$ ad $k_Y\in\N$, such that $\uptau(Y_k) \in [0,T]$ for all $k\ge k_Y$. Therefore, 
\begin{align*}
&\left|g\bigl(Y_k\bigl(\uptau (Y_k)\bigr)\bigr)\E^{\eta(\uptau (Y_k))}-g\bigl(Y\bigl(\uptau (Y)\bigr)\bigl)\E^{\eta(\uptau (Y))}\right|\\&\,\le\,\|g\|_\infty\left|\E^{\eta(\uptau (Y_k))}-\E^{\eta(\uptau (Y))}\right| +\E^{\eta(\uptau (Y))}\left|g\bigl(Y_k\bigl(\uptau (Y_k)\bigr)\bigr)-g\bigl(Y\bigl(\uptau (Y_k)\bigr)\bigl)\right|\\&\ \ \ \ \ +\E^{\eta(\uptau (Y))}\left|g\bigl(Y\bigl(\uptau (Y_k)\bigr)\bigr)-g\bigl(Y\bigl(\uptau (Y)\bigr)\bigl)\right|\\
&\,\le\, \|g\|_\infty\left|\E^{\eta(\uptau (Y_k))}-\E^{\eta(\uptau (Y))}\right| + \E^{\eta(\uptau (Y))}\sup_{0\le t \le T} \left|g\bigl(Y_k(t)\bigr)-g\bigl(Y(t)\bigr)\right|\\&\ \ \ \ \ +\E^{\eta(\uptau (Y))}\left|g\bigl(Y\bigl(\uptau (Y_k)\bigr)\bigr)-g\bigl(Y\bigl(\uptau (Y)\bigr)\bigl)\right|\,.
\end{align*} Clearly, the first and last terms in the above inequality tend to zero as $k$ tends to infinity.
Suppose that $\limsup_{k \to \infty}\sup_{0\le t \le T} |g(Y_k(t))-g(Y(t))|>0$. Then there exist $\epsilon>0$ and sequences $\{k_l\}_{l \in \N} \subseteq \N$ and $\{t_l\}_{l\in \N} \subseteq [0,T]$, such that $\lim_{l\to\infty}t_l = t \in [0,T]$ and $|g(Y_{k_l}(t_l))-g(Y(t_l))|>\epsilon$ for all $l\in\N$. However, this is not possible since $\lim_{l\to\infty}g(Y(t_l)) =g(Y(t))$ and 
\begin{align*}
\lim_{l\to\infty}\left|Y_{k_l}(t_l)-Y(t)\right|&\,\le\, \lim_{l\to\infty}\left|Y_{k_l}(t_l)-Y(t_l)\right|+\lim_{l\to\infty}\left|Y(t_l)-Y(t)\right|\\&\,\le\, \lim_{l\to\infty}\sup_{0\le s \le T}\left|Y_{k_l}(s)-Y(s)\right|+\lim_{l\to\infty}\left|Y(t_l)-Y(t)\right|\\
&\,=\,0\,.
\end{align*}
From this we conclude that 
\begin{equation*}
\lim_{k \to \infty}\left|g\bigl(Y_k\bigl(\uptau (Y_k)\bigr)\bigr)\E^{\eta(\uptau (Y_k))}-g\bigl(Y\bigl(\uptau (Y)\bigr)\bigl)\E^{\eta(\uptau (Y))}\right|\,=\, 0\,,
\end{equation*}
which proves the assertion. \end{proof}

\subsection{The parabolic  problem in \cref{ES1.3}.}  Finally, we discuss the  system of parabolic initial-value problems in \cref{ES1.3}.

\begin{theorem}\label{T4.2} In addition to \textbf{(A1)}-\textbf{(A2)},  assume
	\begin{itemize}
		\item[(a)] $\uppi^0(\bb)=0$ (otherwise we replace $\bb$ by $\bb-\uppi^0(\bb)$ in \cref{ES1.3})
		
		\smallskip
		
		\item [(b)]   $e\in\mathcal{C}(\mathbb{T}_\tau^d\times[n],\R)$ 
		
	\smallskip
		
		\item[(c)] $f,g\in\mathcal{C}(\R^d,\R)$ and $$|f(x)|+|g(x)|\,\le\, K\bigl(1+|x|^\kappa\bigr)$$ for some $K>0$ and $0\le\kappa<2$, and all $x\in\R^d$.
		
	\end{itemize}
	\noindent 		Then, 
	\begin{itemize}
		\item [(i)] \cref{ES1.3} admits a solution $u^\varepsilon\in\mathcal{C}^{1,2}((0,\infty)\times\R^d\times[n],\R)$  (the space of functions $h:(0,\infty)\times\R^d\times[n]\to\R$ satisfying $h(\cdot,x,i)\in\mathcal{C}^1((0,\infty),\R)$ for all $(x,i)\in\R^d\times [n]$, and $h(t,\cdot,i)\in\mathcal{C}^2(\R^d,\R)$ for all $(t,i)\in(0,\infty)\times[n])$ such that \begin{equation}\label{ES4}|u^\varepsilon(t,x,i)|\,\le\, K_{\varepsilon,T} \bigl(1+|x|^{\kappa}\bigr)\qquad\forall\, (t,x,i)\in [0,T]\times\R^d\times[n]\end{equation} for some  $K_{\varepsilon,T}>0$ and all $T>0$
		
		\smallskip
		
		\item[(ii)] 	$ \displaystyle\lim_{\varepsilon \to 0} u^\varepsilon(t,x,i)= u(t,x)$ for all $(x,i,t)\in[0,\infty)\times\R^d\times[n],$ where $u\in\mathcal{C}^{1,2}((0,\infty)\times\R^d,\R)$ is a unique solution to 
		\begin{equation}\label{ES4.9}
		\begin{aligned}
		\partial_t u(t,x)&\,=\,2^{-1}\mathrm{Tr}\bigl(\mathsf{a}\nabla\nabla^\mathrm{T}\bigr) u(t,x)+\mathsf{b}^\mathrm{T}\nabla u(t,x)+\uppi^0(e) u(t,x)+f(x)\\
		u(0,x)&\,=\,g(x)\,,\qquad x\in\R^d\,.
		\end{aligned} 
		\end{equation}
	\end{itemize}
	\end{theorem}
\begin{proof} Existence of a solution $u^\varepsilon\in\mathcal{C}^{1,2}((0,\infty)\times\R^d\times[n],\R)$ to \cref{ES1.3} satisfying \cref{ES4} follows from \cite[Theorem 9.4.3]{Friedman-Book-1964}. Next,  
	according to \cite[Theorem 3.2]{Zhu-Yin-Baran-2015}
	 (see also \cite[Theorem 5.4.1]{Freidlin-Book-1985}),  \cite[Theorem 5.3]{Friedman-Book-1975} and \Cref{R4.3},
	 \begin{equation}\label{ES4.10} \begin{aligned}u^\varepsilon(t,x,i)\,=\,&\mathbb{E}\Big[g\bigl(\X^\varepsilon(x,i;t)\bigr)\E^{\int_0^te(\X^\varepsilon(x,i;s)/\varepsilon,\Lambda(i;s/\varepsilon^2))\D s}\\&+\int_0^tf\bigl(\X^\varepsilon(x,i;s)\bigr)\E^{\int_0^se(\X^\varepsilon(x,i;v)/\varepsilon,\Lambda(i;v/\varepsilon^2))\D v}\D s\Big]\,,\end{aligned}\end{equation}
	and a $\mathcal{C}^{1,2}((0,\infty)\times\R^d,\R)$-unique solution $u(t,x)$ to \cref{ES4.9} exists and
	\begin{align*}u(t,x)\,=\,&\mathbb{E}\Big[g\bigl(\mathrm{W}^{\mathsf{a},\mathsf{b}}(x;t)\bigr)\E^{\uppi^0(e)t}+\int_0^tf\bigl(\mathrm{W}^{\mathsf{a},\mathsf{b}}(x;s)\bigr)\E^{\uppi^0(e)s}\D s\Big]\,.\end{align*}
		We first show that $$\lim_{\varepsilon \to 0}\mathbb{E}\left[g\bigl(\X^\varepsilon(x,i;t)\bigr)\E^{\int_0^te(\X^\varepsilon(x,i;s)/\varepsilon,\Lambda(i;s/\varepsilon^2))\D s}\right]\,=\,\mathbb{E}\left[g\bigl(\mathrm{W}^{\mathsf{a},\mathsf{b}}(x;t)\bigr)\right]\E^{\uppi^0(e)t}\,.$$
	By assumption,
	\begin{align}\label{ES4.11}
	\mathbb{E}\left[\left|g\bigl( \X^\varepsilon(x,i;t)\bigr)\right|^2\right]\,\le\,2K\left(2+\mathbb{E}\left[\bigl|\X^\varepsilon(x,i;t)\bigr)\bigr|^{4}\right]\right)\,.
	\end{align}
	From \Cref{L3.2} it follows that (recall that $\uppi^0(\bb)=0$)
	\begin{equation*}
	\begin{aligned}
\X^\varepsilon(x,i;t) \,=\,&\varepsilon\bar\X^\varepsilon(x/\varepsilon,i;t/\varepsilon^2)\\
 \,=\,&x+\varepsilon\beta\bigl(\bar\X^\varepsilon(x/\varepsilon,i;t/\varepsilon^2),\Lambda(i;t/\varepsilon^2)\bigr)-\varepsilon\beta(x,i)
	\\&+\varepsilon^2\int_0^{t/\varepsilon^2}\bigl((\Id_d-\mathrm{D}\beta)c\bigr)\bigl(\bar\X^\varepsilon(x/\varepsilon,i;s),\Lambda(i;s)\bigr)\D s\\
	&+ \varepsilon\int_0^{t/\varepsilon^2}\bigl((\Id_d-\mathrm{D}\beta)\upsigma\bigr)\bigl(\bar\X^\varepsilon(x/\varepsilon,i;s),\Lambda(i;s)\bigr)\D \B(s)\,.
	\end{aligned}	
	\end{equation*}
	Thus,
	\begin{align*}
	|\X^\varepsilon(x,i;t)|^{4}\,\le\,\bar K\bigg(&|x|^{4}+\varepsilon^{4}+t^{4}\\&+\varepsilon^{4}\left|\int_0^{\varepsilon^{-2}t}\left((\mathbb{I}_d-\mathrm{D}\beta)\upsigma\right)\bigl(\bar\X^\varepsilon(x/\varepsilon,i;s),\Lambda(i;s)\bigr)\D \B(s)\right|^{4}\bigg)\,,
	\end{align*}	
	for some $\bar K>0$ which does not depend on $\varepsilon$. By employing  It\^{o}'s formula and Doob's inequality, we conclude \begin{equation}\label{ES4.12}\mathbb{E}\left[|\X^\varepsilon(x,i;t)|^{4}\right]\,\le\,\tilde K\left(|x|^{4}+\varepsilon^{4}+\max\{t^{2},t^4\}\right)\,,\end{equation}
	for some $\tilde K>0$ which does not depend on $\varepsilon$. Consequently,
	\begin{align*}
	&\left|\mathbb{E}\left[g\bigl(\X^\varepsilon(x,i;t)\bigr)\E^{\int_0^te(\X^\varepsilon(x,i;s)/\varepsilon,\Lambda(i;s/\varepsilon^2))\D s}\right]-\mathbb{E}\left[g\bigl( \X^\varepsilon(x,i;t)\bigr)\right]\E^{ \uppi^0(e)t}\right|\\
	&\,\le\,	\mathbb{E}\left[\left|g\bigl( \X^\varepsilon(x,i;t)\bigr)\right|^2\right]^{1/2}\mathbb{E}\left[\left|\E^{  \int_0^{t}(e(\X^\varepsilon(x,i;s)/\varepsilon,\Lambda(i;s/\varepsilon^2))-\uppi^0(e))\D s}-1\right|\right]^{1/2}\E^{\uppi^0(e)t}\,.
	\end{align*}
		From  \cref{ES4.11,ES4.12} we see that the first term on the right-hand side  is uniformly bounded for $\varepsilon$ on finite intervals. \Cref{ES4.6}, Skorohod representation theorem and dominated convergence theorem imply that the second term on the right-hand side converges to zero as $\varepsilon\to0$.
		It remains to prove that 
	$$\lim_{\varepsilon \to 0}\mathbb{E}\bigl[g\bigl(\X^\varepsilon(x,i;t)\bigr)\bigr]\,=\,\mathbb{E}\bigl[g\bigl(\mathrm{W}^{\mathsf{a},\mathsf{b}}(x,t)\bigr)\bigr]\qquad \forall\, t\ge0\,.$$ Without loss of generality we may assume that $g(x)$ is non-negative. From Skorohod representation theorem and Fatou's lemma we conclude
	$$\liminf_{\varepsilon \to 0}\mathbb{E}\bigl[g\bigl(\X^\varepsilon(x,i;t)\bigr)\bigr]\,\ge\,\mathbb{E}\bigl[g\bigl(\mathrm{W}^{\mathsf{a},\mathsf{b}}(x,t)\bigr)\bigr]\qquad \forall\, t\ge0\,.$$ To prove the reverse inequality we proceed as follows. For any $t\ge0$ we have
	\begin{align*}
	&\limsup_{\varepsilon \to 0}\mathbb{E}\bigl[g\bigl(\X^\varepsilon(x,i;t)\bigr)\bigr]\\ &\,\le\,\limsup_{p\to\infty}\limsup_{\varepsilon \to 0}\mathbb{E}\left[\min\bigl\{g\bigl(\X^\varepsilon(x,i;t)\bigr),p\bigr\}\right]\\&\ \ \ \ +\limsup_{p\to\infty}\limsup_{\varepsilon \to 0}\mathbb{E}\left[g\bigl(\X^\varepsilon(x,i;t)\bigr)\,\mathbb{1}_{\{g(\X^\varepsilon(x,i;t))\ge p\}}\right]\\
	&\,\le\,\limsup_{p\to\infty}\mathbb{E}\left[\min\bigl\{g\bigl(\mathrm{W}^{\mathsf{a},\mathsf{b}}(x,t)\bigr),p\bigr\}\right]\\&\ \ \ \ +\limsup_{p\to\infty}\limsup_{\varepsilon \to 0}\mathbb{E}\left[g\bigl(\X^{\varepsilon}(x,i;t)\bigr)\,\mathbb{1}_{\{g(\X^\varepsilon(x,i;t))\ge p\}}\right]\\
	&\,=\,\mathbb{E}\bigl[g\bigl(\mathrm{W}^{\mathsf{a},\mathsf{b}}(x,t)\bigr)\bigr] +\limsup_{p\to\infty}\limsup_{\varepsilon \to 0}\mathbb{E}\left[g\bigl(\X^\varepsilon(x,i;t)\bigr)\,\mathbb{1}_{\{g(\X^\varepsilon(x,i;t))\ge p\}}\right]\,.
	\end{align*}
	We show next that $$\limsup_{p\to\infty}\limsup_{\varepsilon \to 0}\mathbb{E}\left[g\bigl(\X^\varepsilon(x,i;t)\bigr)\,\mathbb{1}_{\{g(\X^\varepsilon(x,i;t))\ge p\}}\right]\,=\,0\qquad\forall\,t\ge0\,.$$
	We have
	\begin{align*}&\mathbb{E}\left[g\bigl(\X^\varepsilon(x,i;t)\bigr)\,\mathbb{1}_{\{g(\X^\varepsilon(x,i;t))\ge p\}}\right]\\
	&\,\le\,\mathbb{E}\left[\left|g\bigl(\X^\varepsilon(x,i;t)\bigr)\right|^2\right]^{1/2} \left(\Prob\left(g\left(\X^\varepsilon(x,i;t)\right)\ge p\right)\right)^{1/2}\\
	&\,\le\,\frac{1}{p}\,\mathbb{E}\left[\left|g\bigl(\X^\varepsilon(x,i;t)\bigr)\right|^2\right]\,.
	\end{align*}  The assertion  now follows from \cref{ES4.11,ES4.12}.

Finally, we show \begin{align*}&\lim_{\varepsilon \to 0}\mathbb{E}\left[ \int_0^tf\bigl(  \X^\varepsilon(x,i;s)\bigr)\,\E^{\int_0^se(\X^\varepsilon(x,i;v)/\varepsilon,\Lambda(i,v/\varepsilon^2))\D v}\D s\right]\,=\,\mathbb{E}\left[\int_0^tf(\mathrm{W}^{\mathsf{a},\mathsf{b}}(x,s))\,\E^{\uppi^0(e)s}\,\D s\right]\,.\end{align*} From the first part of the proof we see that 
	\begin{align*}&\lim_{\varepsilon \to 0}\mathbb{E}\left[ f\bigl(  \X^\varepsilon(x,i;s)\bigr)\,\E^{\int_0^se(\X^\varepsilon(x,i;v)/\varepsilon,\Lambda(i;v/\varepsilon^2))\D v}\right]\,=\,\mathbb{E}\left[f\bigl(\mathrm{W}^{\mathsf{a},\mathsf{b}}(x,s)\bigr)\,\E^{\uppi^0(e)\, s}\right]\qquad \forall\,s\ge0\,,\end{align*} 
	and
	\begin{align*}&\mathbb{E}\left[ f\bigl( \X^\varepsilon(x,i;s)\bigr)\,\E^{\int_0^se(\X^\varepsilon(x,i;v)/\varepsilon,\Lambda(i;v/\varepsilon^2))\D v}\right]\\
	&\,\le\,\mathbb{E}\left[ \left|f\bigl( \X^\varepsilon(x,i;s)\bigr)\right|^2\right]^{1/2}\mathbb{E}\left[\E^{2\int_0^se(\X^\varepsilon(x,i;v)/\varepsilon,\Lambda(i;v/\varepsilon^2))\,\D v}\right]^{1/2}\\
	&\,\le\, \check{K}\, \bigl(1+|x|^2+\varepsilon^2+\max\{s,s^2\}\bigr)\, \E^{
		\|e\|_\infty s}\,,
	\end{align*} for some $\check{K}>0$ which does not depend on $\varepsilon$.
	The result now follows from the dominated convergence theorem.
	\end{proof}

\begin{remark}\label{R4.3}
	In \cite[Theorem 3.2]{Zhu-Yin-Baran-2015} (and  \cite[Theorem 5.3]{Friedman-Book-1975})
	Feynman-Kac formula for the following backward initial-value problem has been discussed
	\begin{equation}\label{ES4.13} 
	\begin{aligned}
	\partial_tv^\varepsilon(t,x,i)\,=\,&-\bigl(\mathcal{L}_i^{\varepsilon}+\mathcal{Q}\bigr) v^\varepsilon(t,x,i)\\ &-e(x/\varepsilon,i)v^\varepsilon(t,x,i)-f(x)\,,\qquad (t,x,i)\in(0,T)\times\R^d\times[n]\,,\\
	v^\varepsilon(T,x,i)\,=\,& g(x)\,,\qquad (x,i)\in\R^d\times[n]\,,
	\end{aligned} 
	\end{equation}
for fixed $T>0$. 
It has been shown that if \cref{ES4.13}  admits a solution $v^\varepsilon\in\mathcal{C}^{1,2}((0,\infty)\times\R^d\times[n],\R)$   such that $$|v^\varepsilon(t,x,i)|\,\le\, K_{\varepsilon,T} \bigl(1+|x|^{\kappa_{\varepsilon,T}}\bigr)\qquad\forall\, (t,x,i)\in [0,T]\times\R^d\times[n]$$ for some $0\le\kappa_{\varepsilon,T}<2$ and $K_{\varepsilon,T}>0$, and all $T>0$, then 
 \begin{align*}v^\varepsilon(t,x,i)\,=\,&\mathbb{E}\Big[g\bigl(\X^{\varepsilon,t}(x,i;T)\bigr)\E^{\int_t^Te(\X^{\varepsilon,t}(x,i;s)/\varepsilon,\Lambda^t(i;s/\varepsilon^2))\D s}\\&\ \ \ \ +\int_t^Tf\bigl(\X^{\varepsilon,t}(x,i;s)\bigr)\E^{\int_t^se(\X^{\varepsilon,t}(x,i;w)/\varepsilon,\Lambda^t(i;w/\varepsilon^2))\D w}\D s\Big]\,,\qquad 0\le t\le T\,,\end{align*}
 were $\{(\X^{\varepsilon,t}(x,i;s),\Lambda^t(i;s/\varepsilon^2))\}_{s\ge t}$ is a solution to \cref{ES2.1} with $(\X^{\varepsilon,t}(x,i;t),\Lambda^t(i;t/\varepsilon^2))$\linebreak $=(x,i)$. 
However, from this result, and \textit{vice versa}, one can easily deduce that \cref{ES4.10} is a solution to the forward initial-value problem in \cref{ES1.3}. Namely, due to time homogeneity \begin{align*}v^\varepsilon(t,x,i)\,=\,&\mathbb{E}\Big[g\bigl(\X^\varepsilon(x,i;T-t)\bigr)\E^{\int_0^{T-t}e(\X^\varepsilon (x,i;s)/\varepsilon,\Lambda(i;s/\varepsilon^2))\D s}\\&\ \ \ \ \ +\int_0^{T-t}f\bigl(\X^\varepsilon(x,i;s)\bigr)\E^{\int_0^{s}e(\X^\varepsilon(x,i;w)/\varepsilon,\Lambda(i;w/\varepsilon^2))\D w}\D s\Big]\,,\qquad 0\le t\le T\,.\end{align*}
It is straightforward to check  that $v^\varepsilon(T-t,x,i)$ has the representation in \cref{ES4.10} and solves \cref{ES1.3}. 
	\end{remark}

 \section*{Statements and Declarations}
\textbf{Competing Interests:} The author declares that he has no competing interests.

\bibliographystyle{abbrv}
\bibliography{References}

\def\cprime{$'$}
\begin{thebibliography}{10}

\bibitem{Addona-Angiuli-Lorenzi-2019}
D.~Addona, L.~Angiuli, and L.~Lorenzi.
\newblock On invariant measures associated with weakly coupled systems of
  {K}olmogorov equations.
\newblock {\em Adv. Differential Equations}, 24(3-4):137--184, 2019.

\bibitem{Addona-Lorenzi-2023}
D.~Addona and L.~Lorenzi.
\newblock On weakly coupled systems of partial differential equations with
  different diffusion terms.
\newblock {\em Commun. Pure Appl. Anal.}, 22(1):271--303, 2023.

\bibitem{Allaire-2002-Book}
G.~Allaire.
\newblock {\em Shape {O}ptimization by the {H}omogenization {M}ethod}.
\newblock Springer-Verlag, New York, 2002.

\bibitem{Arapostathis-Ghosh-Marcus-1999}
A.~Arapostathis, M.~K. Ghosh, and S.~I. Marcus.
\newblock Harnack's inequality for cooperative weakly coupled elliptic systems.
\newblock {\em Comm. Partial Differential Equations}, 24(9-10):1555--1571,
  1999.

\bibitem{Azunre-2017}
P.~Azunre.
\newblock Bounding the solutions of parametric weakly coupled second-order
  semilinear parabolic partial differential equations.
\newblock {\em Optimal Control Appl. Methods}, 38(4):618--633, 2017.

\bibitem{Zhu-Yin-Baran-2015}
N.~A. Baran, C.~Zhu, and G.~Yin.
\newblock Feynman-{K}ac formulas for regime-switching jump diffusions and their
  applications.
\newblock {\em Stochastics}, 87(6):1000--1032, 2015.

\bibitem{Bensoussan-Lions-Papanicolaou-Book-1978}
A.~Bensoussan, J.~L. Lions, and G.~Papanicolaou.
\newblock {\em Asymptotic analysis for periodic structures}.
\newblock North-Holland Publishing Co., Amsterdam-New York, 1978.

\bibitem{Billingsley-Book-1999}
P.~Billingsley.
\newblock {\em Convergence of probability measures}.
\newblock John Wiley \& Sons, Inc., New York, second edition, 1999.

\bibitem{Boyadzhiev-Kutev-2018}
G.~Boyadzhiev and N.~Kutev.
\newblock Maximum principle for weakly coupled linear non-cooperative systems.
\newblock {\em Pliska Stud. Math.}, 29:37--46, 2018.

\bibitem{Chen-Chen-Tran-Yin-2019}
X.~Chen, Z.-Q. Chen, K.~Tran, and G.~Yin.
\newblock Properties of switching jump diffusions: maximum principles and
  {H}arnack inequalities.
\newblock {\em Bernoulli}, 25(2):1045--1075, 2019.

\bibitem{Chen-Zhao-1994}
Z.-Q. Chen and Z.~Zhao.
\newblock Switched diffusion processes and systems of elliptic equations: a
  {D}irichlet space approach.
\newblock {\em Proc. Roy. Soc. Edinburgh Sect. A}, 124(4):673--701, 1994.

\bibitem{Chen-Zhao-1996}
Z.-Q. Chen and Z.~Zhao.
\newblock Potential theory for elliptic systems.
\newblock {\em Ann. Probab.}, 24(1):293--319, 1996.

\bibitem{Chen-Zhao-1997}
Z.-Q. Chen and Z.~Zhao.
\newblock Harnack principle for weakly coupled elliptic systems.
\newblock {\em J. Differential Equations}, 139(2):261--282, 1997.

\bibitem{Delmonte-Lorenzi-2011}
S.~Delmonte and L.~Lorenzi.
\newblock On a class of weakly coupled systems of elliptic operators with
  unbounded coefficients.
\newblock {\em Milan J. Math.}, 79(2):689--727, 2011.

\bibitem{Down-Meyn-Tweedie-1995}
D.~Down, S.~P. Meyn, and R.~L. Tweedie.
\newblock Exponential and uniform ergodicity of {M}arkov processes.
\newblock {\em Ann. Probab.}, 23(4):1671--1691, 1995.

\bibitem{Durrett-Book-1996}
R.~Durrett.
\newblock {\em Stochastic calculus}.
\newblock CRC Press, Boca Raton, FL, 1996.

\bibitem{Eizenberg-Freidlin-1990}
A.~Eizenberg and M.~Fre\u{\i}dlin.
\newblock On the {D}irichlet problem for a class of second order {PDE} systems
  with small parameter.
\newblock {\em Stochastics Stochastics Rep.}, 33(3-4):111--148, 1990.

\bibitem{Eizenberg-Freidlin-1993-PTRF}
A.~Eizenberg and M.~Fre\u{\i}dlin.
\newblock Averaging principle for perturbed random evolution equations and
  corresponding {D}irichlet problems.
\newblock {\em Probab. Theory Related Fields}, 94(3):335--374, 1993.

\bibitem{Eizenberg-Freidlin-1993-AOP}
A.~Eizenberg and M.~Fre\u{\i}dlin.
\newblock Large deviations for {M}arkov processes corresponding to {PDE}
  systems.
\newblock {\em Ann. Probab.}, 21(2):1015--1044, 1993.

\bibitem{Engler-Lenhart-1991}
H.~Engler and S.~M. Lenhart.
\newblock Viscosity solutions for weakly coupled systems of {H}amilton-{J}acobi
  equations.
\newblock {\em Proc. London Math. Soc. (3)}, 63(1):212--240, 1991.

\bibitem{Ethier-Kurtz-Book-1986}
S.~N. Ethier and T.~G. Kurtz.
\newblock {\em Markov processes}.
\newblock John Wiley \& Sons Inc., New York, 1986.

\bibitem{Eidelman-Book-1969}
S.~D. Eydel{\cprime}man.
\newblock {\em Parabolic systems}.
\newblock North-Holland Publishing Co., Amsterdam-London; Wolters-Noordhoff
  Publishing, Groningen, 1969.

\bibitem{Freidlin-Book-1985}
M.~Fre\u{\i}dlin.
\newblock {\em Functional integration and partial differential equations}.
\newblock Princeton University Press, Princeton, NJ, 1985.

\bibitem{Freidlin-1964}
M.~I. Fre\u{\i}dlin.
\newblock The {D}irichlet problem for an equation with periodic coefficients
  depending on a small parameter.
\newblock {\em Teor. Verojatnost. i Primenen.}, 9:133--139, 1964.

\bibitem{Friedman-Book-1964}
A.~Friedman.
\newblock {\em Partial differential equations of parabolic type}.
\newblock Prentice-Hall, Inc., Englewood Cliffs, N.J., 1964.

\bibitem{Friedman-Book-1975}
A.~Friedman.
\newblock {\em Stochastic differential equations and applications. {V}ol. 1}.
\newblock Academic Press [Harcourt Brace Jovanovich, Publishers], New
  York-London, 1975.

\bibitem{Gilbarg-Trudinger-Book-2001}
D.~Gilbarg and N.~S. Trudinger.
\newblock {\em Elliptic partial differential equations of second order}.
\newblock Springer-Verlag, Berlin, 2001.

\bibitem{Jacod-Shiryaev-2003}
J.~Jacod and A.~N. Shiryaev.
\newblock {\em Limit theorems for stochastic processes}.
\newblock Springer-Verlag, Berlin, second edition, 2003.

\bibitem{Jikov-Kozlov-Oleinik-1994-Book}
V.~V. Jikov, S.~M. Kozlov, and O.~A. Ole{\u\i}nik.
\newblock {\em Homogenization of {D}ifferential {O}perators and {I}ntegral
  {F}unctionals}.
\newblock Springer-Verlag, Berlin, 1994.

\bibitem{Kallenberg-Book-1997}
O.~Kallenberg.
\newblock {\em Foundations of modern probability}.
\newblock Springer-Verlag, New York, 1997.

\bibitem{Kolokoltsov-Book-2011}
V.~N. Kolokoltsov.
\newblock {\em Markov processes, semigroups and generators}, volume~38.
\newblock Walter de Gruyter \& Co., Berlin, 2011.

\bibitem{Kunwai-Zhu-2020}
K.~Kunwai and C.~Zhu.
\newblock On {F}eller and strong {F}eller properties and irreducibility of
  regime-switching jump diffusion processes with countable regimes.
\newblock {\em Nonlinear Anal. Hybrid Syst.}, 38:100946, 21, 2020.

\bibitem{Lazic-Sadric-2022}
P.~Lazi\'c and N.~Sandri\'c.
\newblock On subgeometric ergodicity of regime-switching diffusion processes.
\newblock {\em Nonlinear Analysis: Hybrid systems}, 46:101262, 24, 2022.

\bibitem{Mao-Yuan-Book-2006}
X.~Mao and C.~Yuan.
\newblock {\em Stochastic differential equations with {M}arkovian switching}.
\newblock Imperial College Press, London, 2006.

\bibitem{Meyn-Tweedie-AdvAP-II-1993}
S.~P. Meyn and R.~L. Tweedie.
\newblock Stability of {M}arkovian processes. {II}. {C}ontinuous-time processes
  and sampled chains.
\newblock {\em Adv. in Appl. Probab.}, 25(3):487--517, 1993.

\bibitem{Meyn-Tweedie-AdvAP-III-1993}
S.~P. Meyn and R.~L. Tweedie.
\newblock Stability of {M}arkovian processes. {III}. {F}oster-{L}yapunov
  criteria for continuous-time processes.
\newblock {\em Adv. in Appl. Probab.}, 25(3):518--548, 1993.

\bibitem{Pardoux-Rascanu-Book-2014}
E.~Pardoux and A.~R\u{a}\c{s}canu.
\newblock {\em Stochastic differential equations, backward {SDE}s, partial
  differential equations}.
\newblock Springer, Cham, 2014.

\bibitem{Pinsky-Pinsky-1993}
M.~Pinsky and R.~G. Pinsky.
\newblock Transience/recurrence and central limit theorem behavior for
  diffusions in random temporal environments.
\newblock {\em Ann. Probab.}, 21(1):433--452, 1993.

\bibitem{Protter-Weinberger-Book-1984}
M.~H. Protter and H.~F. Weinberger.
\newblock {\em Maximum principles in differential equations}.
\newblock Springer-Verlag, New York, 1984.

\bibitem{Sirakov-2009}
B.~Sirakov.
\newblock Some estimates and maximum principles for weakly coupled systems of
  elliptic {PDE}.
\newblock {\em Nonlinear Anal.}, 70(8):3039--3046, 2009.

\bibitem{Skorokhod-Book-1989}
A.~V. Skorokhod.
\newblock {\em Asymptotic methods in the theory of stochastic differential
  equations}.
\newblock American Mathematical Society, Providence, RI, 1989.

\bibitem{Sweers-1992}
G.~Sweers.
\newblock Strong positivity in {$C(\overline\Omega)$} for elliptic systems.
\newblock {\em Math. Z.}, 209(2):251--271, 1992.

\bibitem{Tartar-2009-Book}
L.~C. Tartar.
\newblock {\em The {G}eneral {T}heory of {H}omogenization}.
\newblock Springer-Verlag, Berlin; UMI, Bologna, 2009.

\bibitem{Tweedie-1994}
R.~L. Tweedie.
\newblock Topological conditions enabling use of {H}arris methods in discrete
  and continuous time.
\newblock {\em Acta Appl. Math.}, 34(1-2):175--188, 1994.

\bibitem{Xi-Yin-Zhu-2019}
F.~Xi, G.~Yin, and C.~Zhu.
\newblock Regime-switching jump diffusions with non-{L}ipschitz coefficients
  and countably many switching states: existence and uniqueness, {F}eller, and
  strong {F}eller properties.
\newblock In {\em Modeling, stochastic control, optimization, and
  applications}, volume 164 of {\em IMA Vol. Math. Appl.}, pages 571--599.
  Springer, Cham, 2019.

\bibitem{Yin-Zhu-Book-2010}
G.~G. Yin and C.~Zhu.
\newblock {\em Hybrid switching diffusions}.
\newblock Springer, New York, 2010.

\end{thebibliography}

\end{document}